\documentclass[12pt,reqno]{amsart}
\usepackage{amsthm,amsfonts,amssymb,euscript}

\newcommand{\R}{\ensuremath{\mathbb{R}}}
\newcommand{\Z}{\ensuremath{\mathbb{Z}}}

\newcommand{\ep}{\ensuremath{\epsilon}}
\newcommand{\va}{\ensuremath{\varepsilon}}
\newcommand{\ze}{\ensuremath{\zeta}}

\newcommand{\HH}{\ensuremath{\widetilde{H}}}

\newcommand{\DD}{\ensuremath{\mathbf{D}}}

\newtheorem{theorem}{Theorem}[section]
\newtheorem{lemma}[theorem]{Lemma}
\newtheorem{proposition}[theorem]{Proposition}
\newtheorem{corollary}[theorem]{Corollary}

\numberwithin{equation}{section}
\begin{document}

\title{On the stability of certain spin models in $2+1$ dimensions}
\author{I. Bejenaru}
\address{University of Chicago}
\email{bejenaru@math.uchicago.edu}
\author{A. D. Ionescu}
\address{University of Wisconsin -- Madison}
\email{ionescu@math.wisc.edu}
\author{C. E. Kenig}
\address{University of Chicago}
\email{cek@math.uchicago.edu}

\thanks{I. B. was supported in part by NSF grant DMS0738442.
 A. I. was supported in part by a Packard Fellowship. C. K. 
 was supported in part by  NSF grant DMS0456583.}
\begin{abstract}{In this paper we prove large-data local stability theorems for several spin-field models in two dimensions, both in the focusing case (spherical target) and the defocusing case (hyperbolic target).}
\end{abstract}

\maketitle
\tableofcontents

\section{Introduction}\label{intro}

In this paper we consider several 2-dimensional spin models. One of these models is the hyperbolic-elliptic Ishimori system
\begin{equation*}
\partial_t s= s \times (\partial_x^2s-\partial_y^2s)+\partial_xs\cdot\partial_y\ze+\partial_ys\cdot\partial_x\ze,\qquad\Delta\ze= 2s\cdot(\partial_xs\times \partial_ys),
\end{equation*}
where the spin $s$ is defined in an open set of $\mathbb{R}^2\times\mathbb{R}$ and takes values into the $2$-dimensional sphere $\mathbb{S}^2$. The Ishimori system, proposed by Ishimori \cite{Is}, is an integrable topological spin field model. The local and the global regularity properties of the Cauchy problem associated to the Ishimori system have been studied extensively, see for example \cite{HaSa}, \cite{KeNa}, \cite{So} and \cite{Su4}.

We consider in this paper both focusing and defocusing spin systems. To analyze them in a unified, geometric framework we define, for $\mu=\pm 1$, the connected Riemannian manifolds $S_\mu$,
\begin{equation}\label{defi1}
\begin{split}
&S_1=\mathbb{S}^2=\{y=(y_0,y_1,y_2)\in\mathbb{R}^3:y_0^2+y_1^2+y_2^2=1\};\\
&S_{-1}=\mathbb{H}^2=\{y=(y_0,y_1,y_2)\in\mathbb{R}^3:y_0^2-y_1^2-y_2^2=1,\,y_0>0\},
\end{split}
\end{equation}
with the Riemannian structures induced by the Euclidean metric $\mathbf{g}_1=dy_0^2+dy_1^2+dy_2^2$ on $S_1$, respectively the Minkowski metric $\mathbf{g}_{-1}=-dy_0^2+dy_1^2+dy_2^2$ on $S_{-1}$. Thus $S_1$ is the 2-dimensional sphere $\mathbb{S}^2$, while $S_{-1}$ is the 2-dimensional hyperbolic space $\mathbb{H}^2$. Given $\mu=\pm 1$ and two vectors $v={}^t(v_0,v_1,v_2)$ and $w={}^t(w_0,w_1,w_2)$ in $\mathbb{R}^3$, we define their inner product
\begin{equation}\label{defi2}
v\cdot_\mu w=\mathbf{g}_\mu(v,w)={}^tv\cdot\eta_\mu\cdot w=\mu v_0w_0+v_1w_1+v_2w_2,
\end{equation}
where $\eta_\mu=\mathrm{diag}(\mu,1,1)$. We define also the cross product
\begin{equation}\label{defi3}
v\times_\mu w:=\eta_\mu\cdot(v\times w), 
\end{equation}
where $v\times w$ denotes the usual vector product of vectors in $\mathbb{R}^3$. Simple computations show that, for $\mu=\pm 1$ and $v,w\in\mathbb{R}^3$
\begin{equation}\label{products}
\begin{split}
&v\cdot_\mu(v\times_\mu w)=w\cdot_\mu(v\times_\mu w)=0,\\
&(v\times_\mu w)\cdot_\mu(v\times_\mu w)=\mu(v\cdot_\mu v)(w\cdot_\mu w)-\mu(v\cdot_\mu w)^2.
\end{split}
\end{equation}

In this paper we consider the spin-field models on $\R^2\times I$
\begin{equation} \label{Isi}
\partial_t s= s \times_\mu (s_{11}+\ep s_{22})+s_1\ze_2-\ep s_2\ze_1,\qquad\Delta\ze= 2\mu s\cdot_\mu(s_1 \times_\mu s_2),
\end{equation}
where $\ep,\mu\in\{-1,1\}$ and $I\subseteq\R$ is an open interval. The functions $s:\R^2\times I\to S_\mu$ and $\ze:\R^2\times I\to\mathbb{R}$ in \eqref{Isi} are assumed to be sufficiently smooth functions, and $s_1=\partial_xs$, $s_2=\partial_ys$, $s_{11}=\partial_x^2s$, $s_{22}=\partial_y^2s$, $\ze_1=\partial_x\ze$, $\ze_2=\partial_y\ze$, and $\Delta\ze=(\partial_x^2+\partial_y^2)\ze$. 

Spin-field models of this type have been studied in the literature. The pair $(\ep,\mu)=(-1,1)$ corresponds to the hyperbolic-elliptic Ishimori system introduced in \cite{Is}. The defocusing case $\mu=-1$, when the target is the hyperbolic plane $\mathbb{H}^2$, has been introduced and studied in \cite{MaPrSoSo}. The pair $(\ep,\mu)=(1,1)$ corresponds to the incompressible spin fluid system, see \cite{MaPrSoSo}. In this case the spin model \eqref{Isi} becomes
\begin{equation}\label{spiflu}
\partial_t s= s \times \Delta s+s_1\ze_2-s_2\ze_1,\qquad\Delta\ze= 2s\cdot(s_1 \times s_2),
\end{equation}
which is a correction of the classical Heisenberg model (Schr\"{o}dinger map equation)
\begin{equation}\label{Heisen}
\partial_t s= s \times \Delta s.
\end{equation}
This correction was proposed by Volovik \cite{Vo} on physical grounds, for restoration of the correct linear momentum density of the ferromagnets. We emphasize that the mathematical analysis of the Cauchy problem associated to the corrected system \eqref{spiflu} is much simpler than the analysis of the Cauchy problem associated to the Heisenberg model \eqref{Heisen}. The algebraic effect of the correction $s_1\ze_2-s_2\ze_1$ in the right-hand side of \eqref{spiflu} is to cancel the magnetic components of the nonlinearities of the corresponding modified spin system, see section \ref{gauge} for details, which significantly simplifies the analysis of these nonlinearities. This algebraic cancellation is a key feature of all the systems we consider in this paper.
 
We consider ``classical'' solutions of the spin-field models \eqref{Isi}. For $\sigma\geq 1$ we define the spaces of functions
\begin{equation}\label{ho2}
\widetilde{H}^\sigma=\widetilde{H}^\sigma_\mu=\{f\in C^1_b(\R^2:S_\mu):\partial_1f,\partial_2f\in H^{\sigma-1}\},
\end{equation}
where $C^1_b(\R^2:S_\mu)$ denotes the space of bounded $C^1$ functions $f:\R^2\to S_\mu$. For $f,g\in\widetilde{H}^\sigma$ we define
\begin{equation}\label{ho3}
d_\sigma(f,g)=\|f-g\|_{L^\infty}+\sum_{m=1}^2\|\partial_m(f-g)\|_{H^{\sigma-1}},
\end{equation}
and observe that $(\widetilde{H}^\sigma,d_\sigma)$ is a metric space.

We fix, say, $\sigma_0=10$, and consider solutions $s\in C(I:\widetilde{H}^{\sigma_0})$ of \eqref{Isi}, where $I\subseteq\R$ is an open interval. Given such a solution $s$, the function $\ze$ in \eqref{Isi} can be defined as follows: we use the equation $\Delta\ze=2\mu s\cdot_\mu(s_1\times_\mu s_2)\in C(I:L^1\cap L^\infty)$ to define 
\begin{equation}\label{sense}
\ze_1=-R_1\nabla^{-1}[2\mu s\cdot_\mu(s_1\times_\mu s_2)],\qquad\ze_2=-R_2\nabla^{-1}[2\mu s\cdot_\mu(s_1\times_\mu s_2)],
\end{equation}
where the operators $\nabla^{-1}$, $R_1$, and $R_2$ are defined by the Fourier multipliers $|\xi|^{-1}$, $i\xi_1/|\xi|$, and $i\xi_2/|\xi|$ respectively. The functions $\ze_1$ and $\ze_2$ are continuous functions on $\R^2\times I$ and $\ze_1,\ze_2\in C(I:L^p(\R^2))$ for any $p>2$. The function $\ze$ is defined as the unique $C^1$ function on $\R^2\times I$ satisfying
\begin{equation}\label{sense2}
\partial_1\ze=\ze_1,\qquad\partial_2\ze=\ze_2,\qquad\ze(0,0,t)=0.
\end{equation}
Thus $\ze\in C^1(\R^2\times I:\mathbb{R})$ is determined uniquely by $s$ using \eqref{sense} and \eqref{sense2}. In other words, at least for classical solutions $s\in C(I:\widetilde{H}^{\sigma_0})$ the spin model \eqref{Isi} is equivalent to the evolution equation
\begin{equation*}
\partial_t s= s \times_\mu (s_{11}+\ep s_{22})+s_1\ze_2-\ep s_2\ze_1,\quad\ze_m=-R_m\nabla^{-1}[2 \mu s\cdot_\mu(s_1 \times_\mu s_2)].
\end{equation*}

For $q=0,1,2,3$ let
\begin{equation*}
\begin{split}
C^q(I:\widetilde{H}^{\sigma_0})=\{s\in &C(I:\widetilde{H}^{\sigma_0})\cap C^q(\R^2\times I:S_\mu):\\
&\partial^{q'}_ts\in C(I:H^{\sigma_0-2q'}),q'=0,\ldots,q\}.
\end{split}
\end{equation*}
Let 
\begin{equation*}
\mathcal{P}=\{(I,g):I\subseteq\R\text{ is an open interval and }g\in C^3(I:\HH^{\sigma_0})\}
\end{equation*}
with the natural partial order
\begin{equation*}
(I,g)\leq (I',g')\quad\text{ if }I\subseteq I'\text{ and }g'(t)=g(t)\text{ for any }t\in I.
\end{equation*}
Our first theorem is a large-data local regularity result.

\begin{theorem}\label{thm1}
(a) Assume $\sigma_0=10$, $\mu,\ep\in\{-1,1\}$, and $f\in\widetilde{H}^{\sigma_0}$. Then there is a  unique maximal solution $(I(f),s)$ of the initial-value problem
\begin{equation}\label{in1}
\begin{cases}
&\partial_t s= s \times_\mu (s_{11}+\ep s_{22})+s_1\ze_2-\ep s_2\ze_1,\quad\ze_m=-R_m\nabla^{-1}[2 \mu s\cdot_\mu(s_1 \times_\mu s_2)],\\
&s(0)=f,
\end{cases}
\end{equation}
where $I(f)\subseteq \R$ is an open interval and $s\in C^3(I(f):\widetilde{H}^{\sigma_0})$.

(b) Let $I_+(f)=I(f)\cap[0,\infty)$ and $I_-(f)=I(f)\cap (-\infty,0]$. Then
\begin{equation*}
\begin{split}
&\text{ if }I_+(f)\text{ bounded }\text{ then }\| |Ds|\|_{L^4_{x,t}(\R^2\times I_+(f))}=\infty,\\
&\text{ if }I_-(f)\text{ bounded }\text{ then }\| |Ds|\|_{L^4_{x,t}(\R^2\times I_-(f))}=\infty,
\end{split}
\end{equation*}
where
\begin{equation*}
|Ds|:=\big[\sum_{m=1}^2\partial_ms\cdot_\mu\partial_ms\big]^{1/2}.
\end{equation*}
\end{theorem}

In other words, we prove that any classical data admits a unique maximal classical extension as the solution of the spin system \eqref{in1}. This solution extends as long as the critical space-time scattering norm $\||Ds|\|_{L^4}$ stays bounded, where $|Ds|$ is the covariant gradient\footnote{The norm $|Ds|$ is well defined since $v\cdot_\mu v\geq 0$ for any vector $v$ tangent to $S_\mu$ at some $p\in S_\mu$.} of $s$ as defined above. This is similar to well-known theorems on scalar equations, such as the $2$-dimensional $L^2$-critical NLS
\begin{equation*}
(i\partial_t+\Delta_x)\phi=\pm|\phi|^2\phi,\qquad \phi(0)=\phi_0\in L^2.
\end{equation*}

We prove also a stability result. For this we need semidistance functions\footnote{A function $\rho:X\times X\to[0,\infty)$ is a semidistance function on $X$ if $\rho(f,g)=\rho(g,f)$ and $\rho(f,h)\leq\rho(f,g)+\rho(g,h)$ for any $f,g,h\in X$.} $\dot{d}^{1}:\widetilde{H}^{\sigma_0}\times\widetilde{H}^{\sigma_0}\to [0,\infty)$ and $\dot{\rho}_I^1:C^3(I:\HH^{\sigma_0})\times C^3(I:\HH^{\sigma_0})\to[0,\infty)$, defined for any open interval $I\subseteq\R$, which satisfy
\begin{equation}\label{dist1}
\begin{split}
&\dot{d}^1(g,c_Q)=\| |Dg|\|_{L^2},\qquad \dot{\rho}_I^1(g,c_Q)=\| |Dg|\|_{L^\infty_tL^2_x(\R^2\times I)}+\||Dg|\|_{L^4_{x,t}(\R^2\times I)},\\
&\sup_{t\in I}\dot{d}^1(g(t),g'(t))\leq \dot{\rho}_I^1(g,g')\qquad\text{ for any }g,g'\in C^3(I:\HH^{\sigma_0}),
\end{split}
\end{equation}
where, for any $Q\in S_\mu$, $c_Q$ denotes the constant function $c_Q(.)=Q$. We define these semidistance functions precisely in section \ref{proofthm2}, and prove some of their properties in Proposition \ref{distprop}. Intuitively, one could think of $\dot{d}^1(f,f')$ and $\dot{\rho}^1(f,f')$ as nonlinear ways to measure the "distance" between the functions $f,f'$, at a critical level (compare with \eqref{dist1}), in our geometric setting in which the usual "difference" $f'-f$ is not geometrically relevant. Semidistance functions of this type have been used in recent work of Tao \cite{Ta4} on global regularity of wave maps.

Our stability result is the following:  

\begin{theorem}\label{thm2}
Assume that $f\in\widetilde{H}^{\sigma_0}$ and define the maximal solution $(I(f),s)$ as in Theorem \ref{thm1}. Assume that $J\subseteq I(f)$, $0\in J$, is an open interval such that
\begin{equation*}
N(f,J):=\||Df|\|_{L^2(\R^2)}+\||Ds|\|_{L^4(\R^2\times J)}<\infty.
\end{equation*}
Then there is $\overline{\delta}=\overline{\delta}(N(f,J))>0$ with the following  property: if 
\begin{equation*}
f'\in \widetilde{H}^{\sigma_0}\quad\text{ and }\quad \dot{d}^1(f,f')\leq\overline{\delta}
\end{equation*}
then 
\begin{equation*}
J\subseteq I(f'),\quad\text{ and }\quad\dot{\rho}^1_J(s,s')\lesssim_{N(f,J)} \dot{d}^{1}(f,f').
\end{equation*}
\end{theorem}

The identities in \eqref{dist1} and Theorem \ref{thm2} (with $f=c_Q$) can be combined to prove the following small-data global well-posedness result.

\begin{corollary}\label{smalldata}
There is $\delta_0>0$ such that if $f\in \widetilde{H}^{\sigma_0}$ and $\||Df|\|_{L^2}\leq\delta_0$ then the initial-value problem \eqref{in1} admits a unique global solution $s\in C(\R:\widetilde{H}^{\sigma_0})$ and
\begin{equation*}
\| |Ds|\|_{L^\infty_tL^2_x(\R^2\times \R)}+\||Ds|\|_{L^4_{x,t}(\R^2\times \R)}\lesssim \| |Df|\|_{L^2}.
\end{equation*}
In addition, if $f,f'\in \widetilde{H}^{\sigma_0}$, $\||Df|\|_{L^2},\||Df'|\|_{L^2}\in[0,\delta_0]$, and $s,s'\in C(\R:\widetilde{H}^{\sigma_0})$ are the corresponding solutions, then
\begin{equation*}
\dot{\rho}^1_\R(s,s')\lesssim \dot{d}^1(f,f').
\end{equation*}
\end{corollary}

The global regularity part of Corollary \ref{smalldata} has been proved by Chang and Pashaev \cite{ChPa}, at least in the case $\ep=\mu=1$. The proof in \cite{ChPa} relies on perturbative analysis of the ``modified spin system'' (which is derived using the generalized Hasimoto transform) and the key cancellation of the magnetic terms of the nonlinearities of the modified spin system discussed in the paragraph following \eqref{spiflu} and \eqref{Heisen}. These are two of the main ideas we use in this paper as well. See also \cite{So}, \cite{HaSa}, \cite{Su4} for other small-data global regularity results for spin models.

The rest of the paper is organized as follows: in section \ref{gauge} we derive the modified spin system, by taking derivatives $\partial_ms$ of the spin $s$, and decomposing these derivatives in a suitable Coulomb gauge. The idea of using geometric gauges to analyze spin models appears to have been used for the first time in \cite{ChShUh}, in the context of the Schr\"{o}dinger map equation \eqref{Heisen}. This idea was also used in \cite{KeNa}, \cite{NaShVeZe}, and by the authors in \cite{BeIoKe} and \cite{BeIoKeTa}. The entire construction is geometric and can be written invariantly. We prefer however to use an elementary extrinsic point of view in this paper, as in \cite{BeIoKe} and \cite{BeIoKeTa}, in which we exploit the fact that the targets $S_1$ and $S_{-1}$ are isometrically imbedded into the Euclidean space $(\mathbb{R}^3,\mathbf{g}_1)$ and the Minkowski space $(\mathbb{R}^3,\mathbf{g}_{-1})$ respectively. The point of the construction is to link geometric equations, such as the spin model \eqref{in1}, to systems of nonlinear scalar equations, such as the modified spin system in Proposition \ref{summary}. 

In section \ref{modwell} we analyze the modified spin system and prove regularity and stability results for this system, see Propositions \ref{regular} and \ref{wellpo}. Our analysis is based on Strichartz estimates, as well as estimates for the nonlinearities of the modified equations, both at the critical level and the smooth level. These nonlinear estimates are much easier than the corresponding nonlinear estimates in the Schr\"{o}dinger map model, proved in \cite{BeIoKe} and \cite{BeIoKeTa}, due to the absence of magnetic terms in the nonlinearities of the modified spin system. 

In section \ref{proofth1} we prove Theorem \ref{thm1}: we start from the maximal solution of the modified spin system constructed in Proposition \eqref{regular} and construct the maximal solution of the spin system \eqref{in1}, as well as a suitable Coulomb frame, by solving several linear ODE's.

In section \ref{proofthm2} we prove Theorem \ref{thm2}: we define first the critical semidistance functions $\dot{d}^1$ and $\dot{\rho}_I^1$, by taking suitable critical norms of differences of the fields $\psi_m$ constructed using the Coulomb gauge. A nonlinear construction of this type was used recently by Tao \cite{Ta4}, in the setting of wave maps (using the caloric gauge instead of the Coulomb gauge, which is more suitable for the study of wave maps in $2$ dimensions). Then we use the stability result Proposition \ref{wellpo} on the differentiated fields $\psi_m$ to prove Theorem \ref{thm2}. We prove also several additional properties of the semidistance function $\dot{d}^1$ in Proposition \ref{distprop}: invariance under dilations and translations of the domain $\mathbb{R}^2$, invariance under the action of isometries  of the target $S_\mu$, continuity on $\widetilde{H}^{\sigma_0}\times\widetilde{H}^{\sigma_0}$, and a precise description of the set $\{(f,f')\in\widetilde{H}^{\sigma_0}\times\widetilde{H}^{\sigma_0}:\dot{d}^1(f,f')=0\}$.

In section \ref{link} we derive the connection between the Ishimori systems, which correspond to $\ep=-1$, and the Davey-Stewartson equations, starting from our modified spin systems. This connection is well known, see for example \cite{HaSa}, \cite{MaPrSoSo}, \cite{Pa}, or \cite{Su4}, at least in the focusing case $\mu=1$. The analysis in this paper can then be combined with the global analysis of the defocusing Davey-Stewartson II equation, see \cite{Su1}--\cite{Su3}, to give global solutions of the defocusing Ishimori system in the case of large classical data, constant outside a compact set (see Theorem \ref{thmglobal}). It would be desirable, of course, to prove such a large data global regularity result in the defocusing non-integrable case $(\ep,\mu)=(1,-1)$.

In appendix \ref{appendix} we give a simple self-contained proof of the existence and uniqueness (up to the choice of the frame at one point) of a global Coulomb gauge, suitably synchronized in time. This construction is well known, see for example \cite{ChShUh} or \cite{NaShVeZe}. 

\section{The modified spin system}\label{gauge}

In this section we derive the modified spin system, using a Coulomb gauge. Assume in this section that $\mu\in\{-1,1\}$, $I\subseteq \R$ is an open interval, $t_0\in I$, and $s\in C^3(I:\widetilde{H}^{\sigma_0})$. For any point $p\in \R^2\times I$, we fix a small open set  $U_p$ in $\R^2\times I$, $p\in U_p$, and a $C^3$ orthonormal frame in $T_{s}S_\mu$, i.e. two functions $v,w\in C^3(U_p:\R^3)$ such that 
\begin{equation}\label{defi10}
v\cdot_\mu v=1,\quad v\cdot_\mu s=0,\quad w=s\times_\mu v\qquad\text{ in }U_p.
\end{equation}
Easy computations, using also $s\cdot_\mu s=\mu$ and \eqref{products}, show that, in $U_p$
\begin{equation}\label{defi11}
w\cdot_\mu v=w\cdot_\mu s=0,\quad w\cdot_\mu w=1,\quad v\times_\mu w=\mu s,\quad w\times_\mu s=v.
\end{equation}
We define the differentiated variables $\psi_m:U_p\to\mathbb{C}$,
\begin{equation}\label{definitionso}
  \psi_m=v\cdot_\mu\partial_m s
+iw\cdot_\mu\partial_m s,\qquad m=0,1,2,
\end{equation}
where $\partial_0=\partial_t$, and the real connection coefficients $A_m:U_p\to\R$,
\begin{equation}\label{definitions1}
A_m=w\cdot_\mu\partial_m v,\qquad m=0,1,2.
\end{equation}
Clearly $s\cdot_\mu\partial_m s=v\cdot_\mu\partial_m v=w\cdot_\mu\partial_m w=0$, for $m=0,1,2$. Since the vectors $s(p'),v(p'),w(p')$ form an orthonormal basis for $(\R^3,\mathbf{g}_\mu)$, for every $p'\in U_p$, it follows that
\begin{equation}\label{basis}
\begin{cases}
&\partial_m s=v\Re(\psi_m)+w\Im(\psi_m);\\
&\partial_mv=-s \mu\Re(\psi_m)+wA_m;\\
&\partial_mw=-s \mu\Im(\psi_m)-vA_m.
\end{cases}  
\end{equation}

Using \eqref{basis} it is easy to verify that $\psi_m, A_m$ satisfy the curl type relations
\begin{equation}\label{id1}
(\partial_l+iA_l)\psi_m=(\partial_m+iA_m)\psi_l,\qquad m,l=0,1,2.
\end{equation}
Thus with  the notation $\DD_m=\partial_m+iA_m$ we can rewrite 
this as 
\begin{equation}\label{id2}
\DD_l \psi_m= \DD_m \psi_l,\qquad m,l=0,1,2.
\end{equation}
Direct computations using the definitions and \eqref{basis} show that
\begin{equation}\label{id3}
\partial_lA_m-\partial_mA_l=\mu\Im(\psi_l\overline{\psi_m}):=q_{lm},\qquad m,l=0,1,2.
\end{equation}
Thus the curvature of the connection is given by
\begin{equation}\label{id4}
\DD_l\DD_m-\DD_m\DD_l=iq_{lm},\qquad m,l=0,1,2.
\end{equation}

If, in addition, the frame $(v,w)$ can be defined such that the Coulomb condition
\begin{equation*}
\partial_1A_1+\partial_2A_2=0
\end{equation*}
is satisfied in $U_p$, then the identities \eqref{id3} show that
\begin{equation*}
\Delta A_m=\sum_{l=1}^2\partial_lq_{lm},\quad m=0,1,2.
\end{equation*}
We will show first that there is indeed a global $C^3$ frame $(v,w)$, unique up to the choice of $v(0,0,t_0)$, such that the identities above can be formally inverted, in the sense that  $\psi_l\overline{\psi_m}\in C(I:L^1\cap L^\infty)$, $l,m=0,1,2$, and
\begin{equation*}
A_m=-\sum_{l=1}^2\nabla^{-1}R_l[\mu\Im(\psi_l\overline{\psi_m})],\qquad m=0,1,2.
\end{equation*}
The construction of such global Coulomb frames is, of course, well known, see for example \cite{ChShUh} or \cite{NaShVeZe}. We provide all the details here for the sake of completeness. We start with the following simple observation: while the fields $\psi_m$ depend on the choice of $v$, the functions $\psi_l\overline{\psi_m}$, $l,m=0,1,2$ do not depend  on this choice. Indeed, if $(v',w')$ is another frame around the point $p$ then
\begin{equation*}
v'=v\cos\chi+w\sin\chi,\qquad w'=-v\sin\chi+w\cos\chi,
\end{equation*}
for some real-valued function $\chi$, which gives $\psi'_m=e^{-i\chi}\psi_m$. Therefore, given $s\in C^3(I:\widetilde{H}^{\sigma_0})$ we can define $9$ canonical functions $\psi_l\overline{\psi_m}$, $m,l=0,1,2$. 

Since $s(t)$ is bounded for any $t\in I$, the functions $v(.,.,t),w(.,.,t):\R^2\to\R^3$ are bounded for any $t\in I$, thus  $\psi_l\overline{\psi_m}\in C(I:L^1\cap L^\infty)$. Moreover, if we work only with local $C^3$ frames $v,w$ with derivatives bounded uniformly on compact subintervals $J\subseteq I$, we deduce that $\partial_n(\psi_l\overline{\psi_m})\in C(I:L^1\cap L^\infty)$, $n,m,l=0,1,2$. To summarize, 
\begin{equation}\label{nj15}
\psi_l\overline{\psi_m}\in C^2(\R^2\times I),\quad\psi_l\overline{\psi_m},\partial_n(\psi_l\overline{\psi_m})\in C(I:L^1\cap L^\infty),\quad n,m,l=0,1,2.
\end{equation}
Thus, we can define $C^1$ functions $\widetilde{A}_0,\widetilde{A}_1,\widetilde{A}_2:\mathbb{R}^2\times I\to\mathbb{R}$ by the formulas
\begin{equation}\label{nj14}
\widetilde{A}_m=-\sum_{l=1}^2\nabla^{-1}R_l[\mu\Im(\psi_l\overline{\psi_m})].
\end{equation}
We show now that there are global $C^3$ Coulomb frames $(v,w)$ on $\R^2\times I$ such that the coefficients $A_m=w\cdot_\mu\partial_m v$, $m=0,1,2$ (see \eqref{definitions1}) agree with the coefficients $\widetilde{A}_m$ defined in \eqref{nj14}. More  precisely:

\begin{proposition}\label{vwoni}
Assume $I\subseteq\R$ is an open interval, $t_0\in I$, and $s\in C^3(I:\widetilde{H}^{\sigma_0})$. Assume $Q\in\R^3$, $Q\cdot_\mu s(0,0,t_0)=0$, $Q\cdot_\mu Q=1$. Then there are unique functions $v,w\in C^3(\R^2\times  I:\R^3)$ with the properties
\begin{equation}\label{nl16}
v\cdot_\mu s=0,\quad v\cdot_\mu v=1,\quad w=s\times_\mu v,\quad v(0,0,t_0)=Q,
\end{equation}
and
\begin{equation}\label{nl17}
\widetilde{A}_m=w\cdot_\mu\partial_m v,\qquad m=0,1,2.
\end{equation}
In addition, if $\psi_m=v\cdot_\mu\partial_m s+iw\cdot_\mu\partial_m s$, $m=0,1,2$, then $\psi_m\in C(I:H^4)$ and
\begin{equation}\label{nl18}
(\partial_l+i\widetilde{A}_l)\psi_m=(\partial_m+i\widetilde{A}_m)\psi_l,\qquad m,l=0,1,2.
\end{equation}
and, for $m=0,1,2$,
\begin{equation}\label{nl18.5}
\begin{cases}
&\partial_m s=v\Re(\psi_m)+w\Im(\psi_m);\\
&\partial_mv=-s \mu\Re(\psi_m)+w\widetilde{A}_m;\\
&\partial_mw=-s \mu\Im(\psi_m)-v\widetilde{A}_m.
\end{cases}  
\end{equation}
\end{proposition}

We provide a complete proof of Proposition \ref{vwoni} in the appendix. 

We convert now the spin system \eqref{in1} into a system of equations
involving the fields $\psi_m$. Assume $I\subseteq R$ is an open set and $s\in C(I:\widetilde{H}^{\sigma_0})$ 
satisfies the equation
\begin{equation}\label{nl19}
\partial_0 s= s \times_\mu (s_{11}+\ep s_{22})+s_1\ze_2-\ep s_2\ze_1,\quad\ze_m=-R_m\nabla^{-1}[2 \mu s\cdot_\mu(s_1 \times_\mu s_2)].
\end{equation} 
We fix a global Coulomb frame $(v,w)$ as in Proposition \ref{vwoni} and define the fields $\psi_m$ and the connection coefficients $A_m=\widetilde{A}_m$, $m=0,1,2$, such that the identities  \eqref{defi10}-\eqref{id4} and \eqref{nj14} hold in $\R^2\times I$. Using \eqref{basis} we have 
\begin{equation*}
\begin{split}
2 \mu s\cdot_\mu(s_1 \times_\mu s_2)&=2 \mu s \cdot_\mu [(v\Re (\psi_1)  + w\Im ( \psi_1) ) \times_\mu (v\Re (\psi_2)  + w\Im ( \psi_2) )]\\
&= 2\mu( \Re (\psi_1) \Im ( \psi_2) - \Im ( \psi_1) \Re (\psi_2))\\
&=-2 q_{12}.
\end{split}
\end{equation*}
It follows that
\begin{equation} \label{xifor}
\ze_1 =2R_1\nabla^{-1}(q_{12}), \qquad \ze_2= 2 R_2\nabla^{-1}(q_{12}).
\end{equation}
Using \eqref{definitionso} and \eqref{basis} into the first equation in \eqref{nl19} we compute
\begin{equation*}
\begin{split}
\partial_0s&= s \times_\mu (s_{11}+\ep s_{22})+ s_1\ze_2-\ep s_2\ze_1\\
    &= s \times_\mu\big[ v(\partial_1 \Re(\psi_1)-A_1\Im(\psi_1))+w(\partial_1\Im(\psi_1)+A_1\Re(\psi_1))-\mu s|\psi_1|^2 \big] \\
    &+\ep s \times_\mu\big[v(\partial_2\Re(\psi_2)-A_2\Im(\psi_2))+w(\partial_2\Im(\psi_2)+A_2\Re(\psi_2))-\mu s|\psi_2|^2\big] \\
&+(v\Re (\psi_1)+ w\Im ( \psi_1))\ze_2-\ep (v\Re (\psi_2)+w\Im ( \psi_2))\ze_1\\
&=v\big[-\partial_1\Im(\psi_1)-\ep\partial_2\Im(\psi_2)-A_1\Re(\psi_1)-\ep A_2\Re(\psi_2)+\ze_2\Re(\psi_1)-\ep\ze_1\Re(\psi_2)\big]\\
&+w\big[\partial_1 \Re(\psi_1)+\ep\partial_2 \Re(\psi_2)-A_1\Im(\psi_1)-\ep A_2\Im(\psi_2)+\ze_2\Im(\psi_1)-\ep\ze_1\Im(\psi_2)\big].
\end{split}
\end{equation*}
Thus
\begin{equation} \label{t43}
\psi_0 = v\cdot_\mu \partial_t s+ i w\cdot_\mu\partial_t s
= i ( \DD_1 \psi_1 +\ep\DD_2 \psi_2) +\ze_2\psi_1-\ep \ze_1 \psi_2.
\end{equation}
Using the \eqref{id2} and \eqref{id4}, for $m=1,2$ we derive
\begin{equation} \label{covaux}
\begin{split}
&\DD_0 \psi_m = \DD_m \psi_0 = i \DD_m  ( \DD_1 \psi_1+\ep\DD_2 \psi_2)  + \DD_m (\ze_2\psi_1-\ep \ze_1 \psi_2) \\
&= i (\DD_1^2+\ep\DD_2^2) \psi_m - q_{m1} \psi_1 -\ep q_{m2}\psi_2 +\ze_2\DD_1 \psi_m -\ep\ze_1\DD_2 \psi_m+\psi_1\partial_m\ze_2 -\ep\psi_2\partial_m \ze_1. 
\end{split}
\end{equation}
By direct computation
\begin{equation*}
\begin{split}
&i (\DD_1^2 +\ep\DD_2^2) \psi\\
&= i (\partial_1^2+\ep\partial_2^2)\psi - 2A_1 \partial_1 \psi -2\ep A_2 \partial_2 \psi 
- \psi\partial_1 A_1 -\ep \psi\partial_2 A_2 - i (A_1^2 +\ep A_2^2) \psi.
\end{split}
\end{equation*}
Hence, the equations for $\psi_m$, $m=1,2$, are
\begin{equation} \label{iscoord}
\begin{split}
&\partial_0 \psi_m - i (\partial_1^2+\ep\partial_2^2)\psi_m = (\ze_2-2A_1)\partial_1\psi_m-\ep(\ze_1+2A_2)\partial_2\psi_m-q_{m1} \psi_1-\ep q_{m2}\psi_2\\
&+\psi_1\partial_m\ze_2-\ep\psi_2\partial_m \ze_1+\psi_m(-\partial_1A_1-\ep\partial_2A_2-iA_1^2-i\ep A_2^2+i\ze_2A_1-i\ep\ze_1A_2-iA_0).
\end{split}
\end{equation}

Using now the identities \eqref{nj14} (recall $A_m=\widetilde{A}_m$) and \eqref{xifor}, we notice that the magnetic terms $(\ze_2-2A_1)\partial_1\psi_m$ and $-\ep(\ze_1+2A_2)\partial_2\psi_m$ in the right-hand side of \eqref{iscoord} vanish. This cancellation, which is due to the correction terms in the spin field models \eqref{Isi}, is the main reason for the simplicity of these models compared to the Heisenberg model.  

To finish our computation, we observe that we have formulas, in terms of the functions $\psi_1,\psi_2$, of all the functions in the right-hand side of \eqref{iscoord}, with the exception of $A_0$. To compute $A_0$, using \eqref{nj14},
\begin{equation} \label{da0}
A_0 = -\mu\sum_{l=1}^2 \nabla^{-1}R_l\Im(\psi_l \overline{\psi_0}).
\end{equation}
Using \eqref{id1}, \eqref{t43} (with $\ze_1=-2A_2$, $\ze_2=2A_1$), and the identity $\overline{\psi_l}\cdot \DD_m\psi_m=\partial_m(\overline{\psi_l}\psi_m)-\psi_m\cdot\overline{\DD_m\psi_l}$, we derive
\begin{equation*}
\begin{split}
&\Im(\psi_l\overline{\psi_0})=- \Re(\overline{\psi_l}\cdot (\DD_1\psi_1+\ep\DD_2 \psi_2)) - \Im( \overline{\psi_l} (2 A_1 \psi_1+2\ep A_2 \psi_2)) \\
&=- \partial_1 \Re(\overline{\psi_l} \psi_1)-\ep\partial_2 \Re(\overline{\psi_l} \psi_2)  + \Re(\psi_1\overline{\DD_1\psi_l})+\ep\Re(\psi_2 \overline{\DD_2\psi_l}) - 2 \Im( \overline{\psi_l} (A_1 \psi_1+\ep A_2 \psi_2))\\
&=- \partial_1 \Re(\overline{\psi_l} \psi_1)-\ep\partial_2 \Re(\overline{\psi_l}\psi_2) +\frac{1}{2}\partial_l\big( |\psi_1|^2+\ep|\psi_2|^2 \big)- 2  \Im( \overline{\psi_l} (A_1 \psi_1+\ep A_2 \psi_2)).
\end{split}
\end{equation*}
Thus
\begin{equation} \label{a0f}
\begin{split}
A_0 &=\mu\sum_{m,l=1}^2 \ep^{m+1} R_lR_m\big(\Re(\overline{\psi_l}\psi_m)\big)+\frac{\mu}{2}(|\psi_1|^2+\ep|\psi_2|^2 \big)\\
&+2\mu\sum_{m,l=1}^2 \ep^{m+1} |\nabla|^{-1} R_l \Im( A_m \psi_m \overline{\psi_l}).
\end{split}
\end{equation}
We summarize our results so far in the following proposition:

\begin{proposition}\label{summary} Assume $s\in C(I:\widetilde{H}^{\sigma_0})$ is a solution of the equation \eqref{nl19}. Assume that $v,w$ is a Coulomb frame on $\R^2\times I$ as in Proposition \ref{vwoni}, and let
\begin{equation}\label{bas1}
\begin{split}
&\psi_m=v\cdot_\mu\partial_m s
+iw\cdot_\mu\partial_m s, \qquad A_m=w\cdot_\mu\partial_m v,\\
&q_{lm}=\partial_lA_m-\partial_mA_l=\mu\Im(\psi_l\overline{\psi_m}),
\end{split}
\end{equation}
for $m,l=0,1,2$. Then $\psi_1,\psi_2\in C(I:H^4)$ and
\begin{equation}\label{bas2}
\begin{split}
&A_2=-\nabla^{-1}R_1(q_{12}),\qquad A_1=\nabla^{-1}R_2(q_{12}),\\
&A_0=\mu\sum_{m,l=1}^2 \ep^{m+1}\big[R_lR_m\big(\Re(\overline{\psi_l}\psi_m)\big)+2|\nabla|^{-1} R_l \Im( A_m \psi_m \overline{\psi_l})\big]+\frac{\mu}{2}\sum_{m=1}^2\ep^{m+1}|\psi_m|^2.
\end{split}
\end{equation}
In addition, the fields $\psi_m$, $m=0,1,2$, satisfy the equations
\begin{equation} \label{bas3}
\begin{split}
&i \partial_t \psi_m +(\partial_1^2+\ep\partial_2^2) \psi_m = i\mathcal{N}_m,\\
&\mathcal{N}_m =-iA_0\psi_m+\sum_{l=1}^2\ep^{l+1}\big[\psi_l(-q_{ml}+2\partial_mA_l)+\psi_m(-\partial_lA_l+iA_l^2)\big].
\end{split}
\end{equation}
and
\begin{equation}\label{bas4}
\psi_0=i\partial_1\psi_1+i\ep\partial_2\psi_2+A_1\psi_1+\ep A_2\psi_2.
\end{equation}
\end{proposition}

\section{Regularity and stability of the modified spin system}\label{modwell}

In this section we analyze the modified spin system constructed  in Proposition \ref{summary}. We will prove first a large-data local regularity result.

\begin{proposition}\label{regular} Assume $\phi=(\phi_1,\phi_2)\in H^{\sigma_0-1}\times H^{\sigma_0-1}$. 

(a) There is a unique maximal open interval $I(\phi)$, $0\in I(\phi)$, and a unique solution $\psi=(\psi_1,\psi_2)\in C(I(\phi):H^{\sigma_0-1})\times C(I(\phi):H^{\sigma_0-1})$ of the system of equations
\begin{equation}\label{tot01}
i \partial_t \psi_m +(\partial_1^2+\ep\partial_2^2) \psi_m = i\mathcal{N}_m,\qquad\psi_m(0)=\phi_m,
\end{equation}
where
\begin{equation}\label{tot1}
\begin{split}
&\mathcal{N}_m =-iA_0\psi_m+\sum_{l=1}^2\ep^{l+1}\big[\psi_l(\partial_lA_m+\partial_mA_l)+\psi_m(-\partial_lA_l+iA_l^2)\big],\\
&q_{12}=\mu\Im(\psi_1\overline{\psi_2}),\qquad A_2=-\nabla^{-1}R_1(q_{12}),\qquad A_1=\nabla^{-1}R_2(q_{12}),\\
&A_0=\mu\sum_{m,l=1}^2 \ep^{m+1}\big[R_lR_m\big(\Re(\overline{\psi_l}\psi_m)\big)+2|\nabla|^{-1} R_l \Im( A_m \psi_m \overline{\psi_l})\big]+\frac{\mu}{2}\sum_{m=1}^2\ep^{m+1}|\psi_m|^2.
\end{split}
\end{equation}

(b) Let $I_+(\phi)=I(\phi)\cap[0,\infty)$, $I_-(\phi)=I(\phi)\cap(-\infty,0]$. Then
\begin{equation*}
\begin{split}
&\text{ if }I_+(\phi)\text{ bounded }\text{ then }\sum_{m=1}^2\|\psi_m\|_{L^4_{x,t}(\R^2\times I_+(\phi))}=\infty,\\
&\text{ if }I_-(\phi)\text{ bounded }\text{ then }\sum_{m=1}^2\|\psi_m\|_{L^4_{x,t}(\R^2\times I_-(\phi))}=\infty.
\end{split}
\end{equation*}

(c) Assume, in addition, that the compatibility condition 
\begin{equation}\label{tot2}
(\partial_1+iA_1)\psi_2=(\partial_2+iA_2)\psi_1
\end{equation}
holds on $\R^2\times\{0\}$. Then the identities (compare with \eqref{id2} and \eqref{id4})
\begin{equation}\label{ea9}
\DD_l\psi_m=\DD_m\psi_l,\qquad \partial_lA_m-\partial_mA_l=\mu\Im(\psi_l\overline{\psi_m}),\qquad m,l=0,1,2,
\end{equation} 
hold in $\R^2\times I(\phi)$, where $\DD_m=\partial_m+iA_m$ and
\begin{equation*}
\psi_0=i(\DD_1\psi_1+\ep\DD_2\psi_2)+2A_1\psi_1+2\ep A_2\psi_2.
\end{equation*}
\end{proposition}

We will also prove a stability result.

\begin{proposition}\label{wellpo}
Assume $\phi\in H^{\sigma_0-1}\times H^{\sigma_0-1}$ and construct the maximal extension $(I(\phi),\psi)$ as in Proposition \ref{regular}. Assume  $J\subseteq I(\phi)$ is a compact interval, $0\in J$. Let
\begin{equation*}
N_{\psi,J}=\| |\psi|\|_{L^4_{x,t}(\R^2\times J)}+\||\phi|\|_{L^2},\quad|\psi|^2=|\psi_1|^2+|\psi_2|^2,\,\,|\phi|^2=|\phi_1|^2+|\phi_2|^2.
\end{equation*}
Then there is $\delta_0=\delta_0(N_{\psi,J})$ with the following property: if $\phi'\in H^{\sigma_0-1}\times H^{\sigma_0-1}$ and
\begin{equation*}
\||\phi'-\phi|\|_{L^2(\R^2)}=\delta\leq\delta_0
\end{equation*}
then
\begin{equation*}
J\subseteq I(\phi')\quad\text{ and }\quad\sum_{m=1}^2\|\mathcal{N}'_m-\mathcal{N}_m\|_{(L^1_tL^2_x+L^{4/3}_{x,t})(\R^2\times J)}\lesssim_{N_{\psi,J}}\delta.
\end{equation*}
\end{proposition}

\subsection{Linear and nonlinear estimates}\label{estimates} The linear evolution associated to the modified spin system is
\begin{equation} \label{li}
\begin{split}
i \partial_t u + (\partial_1^2 + \epsilon \partial_2^2) u = g.
\end{split}
\end{equation}
It was established in \cite{GhSa} that this linear evolution enjoys dispersive properties similar to those of the Schr\"odinger evolution, in the sense that the standard Strichartz estimates hold.

\begin{lemma} If $(p,p')$ and $(q,q')$ are dual pairs, $\frac{1}{p} + \frac{1}{q}=\frac12$, $2 < p \leq \infty$, $I\subseteq\R$ is an open interval and $t_0\in I$, then for any solution of \eqref{li} on $\R^2\times I$, 
\begin{equation} \label{listr}
\| u \|_{(L^\infty_tL^2_x\cap L^p_t L^q_x)(\R^2\times I)} \lesssim_p \| u(t_0) \|_{L^2} + \| g\|_{(L^1_tL^2_x+L^{p'}_t L^{q'}_x)(\R^2\times I)}.
\end{equation}
\end{lemma}

To control higher regularity norms it is convenient to use Littlewood-Paley decompositions. Given an open interval $I\subseteq\mathbb{R}$ we define the Banach  spaces $X^\sigma(I)$, $\sigma\geq 0$,
\begin{equation} \label{defx}
\begin{split}
X^\sigma(I)&=\{\phi\in C(I:H^\sigma):\\
&\| \phi \|_{X^\sigma(I)} =\big( \sum_{k \in \Z} (1+2^{2\sigma k}) \| P_k \phi \|^2_{(L^\infty_t L^2_x\cap L^4_{x,t})(\R^2\times  I)} \big)^{1/2}<\infty\},
\end{split}
\end{equation}
where $P_k$ denote smooth Littlewood-Paley projections.\footnote{More precisely, the operators $P_k$, $k\in\mathbb{Z}$, are defined by the Fourier multipliers $\xi\to\chi_k(|\xi|)$, where $\chi_k(\mu)=\eta_0(\mu/2^k)-\eta_0(\mu/2^{k-1})$ and $\eta_0:\R\to[0,1]$ is an even smooth function supported in the interval $[-8/5,8/5]$ and equal to $1$ in the interval $[-5/4,5/4]$.} We will measure the nonlinearities $\mathcal{N}_m$ in the normed spaces $Y_{T}^\sigma$ defined by the norm
\begin{equation}\label{defy}
\| \phi \|_{Y^\sigma(I)} = \big( \sum_{k \in \Z} (1+2^{2\sigma k}) \| P_k \phi \|^2_{L^1_tL^2_x+L^{4/3}_{x,t}(\R^2\times I)} \big)^{1/2}.
\end{equation}
It follows from \eqref{listr} that if $I\subseteq\R$ is  an open interval, $t_0\in I$, and $i \partial_t u + (\partial_1^2 + \epsilon \partial_2^2) u = g$ on $\R^2\times I$, then
\begin{equation}\label{listr2}
\begin{split}
&\|u\|_{(L^\infty_t L^2_x\cap L^4_{x,t})(\R^2\times  I)}\lesssim\|u(t_0)\|_{L^2}+\|g\|_{(L^1_tL^2_x+L^{4/3}_{x,t})(\R^2\times I)},\\
&\| u \|_{X_I^\sigma} \lesssim\| u(t_0)\|_{H^\sigma} + \|g\|_{Y^\sigma_I},\qquad\text{  for any }\sigma\in[0,\sigma_0-1].
\end{split}
\end{equation}
Using the Littlewood-Paley square function estimate we notice that
\begin{equation}\label{ea2}
\begin{split}
&\|\phi\|_{(L^\infty_t L^2_x\cap L^4_{x,t})(\R^2\times  I)}\lesssim \|\phi\|_{X^0(I)},\\
&\|\phi\|_{Y^0(I)}\lesssim \|\phi\|_{(L^1_tL^2_x+L^{4/3}_{x,t})(\R^2\times I)}.
\end{split}
\end{equation}

We estimate now the nonlinearities $\mathcal{N}_m$.

\begin{proposition} \label{propnon} Assume that $I\subseteq\R$ is an open interval, $\psi_m\in C(I:H^4)$, $m=1,2$, and define $\mathcal{N}_m$ as in \eqref{tot1}. Assume also that
\begin{equation}\label{tobootstrap}
\sum_{m=1}^2\| \psi_m\|_{(L^4_{x,t}\cap L^{12}_tL^{12/5}_x)(\R^2\times I)}\leq a\leq 1.
\end{equation}
Then
\begin{equation}\label{nonlin1}
\sum_{m=1}^2\|\mathcal{N}_m \|_{(L^1_tL^2_x+L^{4/3}_{x,t})(\R^2\times I)} \lesssim a^3,
\end{equation}
and, if $\psi_m\in X^{\sigma_0-1}(I)$, $m=1,2$, then 
\begin{equation}\label{nonlin2}
\sum_{m=1}^2\|\mathcal{N}_m \|_{Y^{\sigma_0-1}(I)} \lesssim a^2\sum_{m=1}^2 \|\psi_m\|_{X^{\sigma_0-1}(I)}.
\end{equation}
Assume, in addition, that $\psi'_m\in C(I:H^4)$, $m=1,2$, also satisfy \eqref{tobootstrap}, $\mathcal{N}'_m$ are defined as in \eqref{tot1}, and let
\begin{equation*}
b=\sum_{m=1}^2\|\psi'_m-\psi_m\|_{(L^{12}_tL^{12/5}_x\cap L^4_{x,t})(\R^2\times I)}\leq 2a.
\end{equation*}
Then
\begin{equation}\label{nonlin3}
\sum_{m=1}^2\|\mathcal{N}'_m-\mathcal{N}_m \|_{(L^1_tL^2_x+L^{4/3}_{x,t})(\R^2\times I)} \lesssim a^2b,
\end{equation}
and, if $\psi'_m\in X^{\sigma_0-1}(I)$, $m=1,2$, then
\begin{equation}\label{nonlin4}
\begin{split}
\sum_{m=1}^2\|\mathcal{N}'_m-\mathcal{N}_m \|_{Y^{\sigma_0-1}(I)}&\lesssim a^2\sum_{m=1}^2 \|\psi'_m-\psi_m\|_{X^{\sigma_0-1}(I)}\\
&+ab\sum_{m=1}^2(\|\psi_m\|_{X^{\sigma_0-1}(I)}+\|\psi'_m\|_{X^{\sigma_0-1}(I)}).
\end{split}
\end{equation}
\end{proposition}

The rest of this subsection is concerned with the proof of Proposition \ref{propnon}. The bounds \eqref{nonlin1} and \eqref{nonlin2} clearly follow from \eqref{nonlin3} and \eqref{nonlin4} with $\psi'_m=0$. For simplicity of notation, let $L^p_tL^q_x=L^p_tL^q_x(\R^2\times I)$ in the rest of the proof. To prove the bound \eqref{nonlin4} we work with frequency envelopes. For $\sigma=\sigma_0-1$ we define, for any $k\in\mathbb{Z}$, 
\begin{equation} \label{env}
a_k(\sigma) = \sum_{m=1}^2[\sup_{k' \in \Z} 2^{-|k-k'|/20}2^{\sigma k'}(\| P_{k'} \psi_m \|_{L^\infty_tL^2_x\cap L^4_{x,t}}+\| P_{k'} \psi'_m \|_{L^\infty_tL^2_x\cap L^4_{x,t}})],
\end{equation}
and
\begin{equation} \label{envdif}
b_k(\sigma) = \sum_{m=1}^2[\sup_{k' \in \Z} 2^{-|k-k'|/20}2^{\sigma k'}\| P_{k'} (\psi'_m-\psi_m) \|_{L^\infty_tL^2_x\cap L^4_{x,t}}].
\end{equation}
The envelope coefficients satisfy the inequalities
\begin{equation}\label{psienv2}
\sum_{k\in\Z}a_k(\sigma)^2 \lesssim \sum_{m=1}^2(\|\psi_m\|^2_{X^{\sigma}(I)}+\| \psi'_m\|^2_{X^{\sigma}(I)}),\qquad\sum_{k\in\Z}b_k(\sigma)^2 \lesssim \sum_{m=1}^2\|\psi'_m-\psi_m\|^2_{X^{\sigma}(I)},
\end{equation}
\begin{equation}\label{psienv3}
a_{k'}(\sigma) \leq 2^{|k-k'|/20} a_k(\sigma),\qquad b_{k'}(\sigma) \leq 2^{|k-k'|/20} b_k(\sigma),
\end{equation}
and
\begin{equation} \label{psienv}
\begin{split}
&\| P_k \psi_m \|_{L^\infty_tL^2_x\cap L^4_{x,t}}+\|P_k\psi'_m \|_{L^\infty_tL^2_x\cap L^4_{x,t}}\lesssim 2^{-\sigma k} a_k(\sigma),\\
&\| P_k(\psi'_m-\psi_m)\|_{L^\infty_tL^2_x\cap L^4_{x,t}}\lesssim 2^{-\sigma k} b_k(\sigma),
\end{split}
\end{equation}
for $m=1,2$ and $k,k'\in\mathbb{Z}$. The following simple lemma will be used several times in this section.

\begin{lemma}\label{al9}
Assume $f,g\in C(I:L^2)$ and $p_i,q_i\in[1,\infty]$, $i=1,2,3$, satisfy $1/p_1+1/p_2=1/p_3$ and $1/q_1+1/q_2=1/q_3$. For $k\in\Z$ let
\begin{equation*}
\rho_k=\sum_{|k'-k|\leq 20}\|P_{k'}f\|_{L^{p_1}_tL^{q_1}_x},\quad\nu_k=\sum_{|k'-k|\leq 20}\|P_{k'}g\|_{L^{p_2}_tL^{q_2}_x}.
\end{equation*}
Let
\begin{equation*}
\rho =\|f\|_{L^{p_1}_tL^{q_1}_x} , \qquad \nu = \|g\|_{L^{p_2}_tL^{q_2}_x}.
\end{equation*}
Then, for any $k \in \Z$, 
\begin{equation} \label{resaux}
\| P_k (fg) \|_{L^{p_3}_tL^{q_3}_x} \lesssim \rho\sum_{k'\geq k}\nu_{k'} + \nu\sum_{k' \geq k}\rho_{k'}.
\end{equation}
\end{lemma}

\begin{proof} [Proof of Lemma \ref{al9}] We decompose
\begin{equation}\label{mf60}
\begin{split}
P_k(fg)&= P_k (P_{\leq k-4}f \cdot P_{[k-3,k+3]} g) + P_k (P_{[k-3,k+3]} f \cdot P_{\leq k-4} g)\\
&+\sum_{k_1,k_2\geq k-3,|k_1-k_2|\leq 8}  P_k(P_{k_1} f \cdot P_{k_2} g),
\end{split}
\end{equation}
where, for any interval $J\subseteq\R$, $P_J=\sum_{j\in J}P_j$ and $P_{\leq j}=P_{(-\infty,j]}$. The bound \eqref{resaux} follows since $\|P_{\leq j}f\|_{L^{p_1}_tL^{q_1}_x}\lesssim \rho$ and $\|P_{\leq j}g\|_{L^{p_2}_tL^{q_2}_x}\lesssim \nu$ for any $j\in\mathbb{Z}$.
\end{proof}

We analyze now the coefficients of $\psi_m$ in the nonlinearities $\mathcal{N}_m$. Recall the formulas, see Proposition \ref{summary},
\begin{equation}\label{t52}
\begin{split}
&q_{12}=\mu\Im(\psi_1\overline{\psi_2}),\qquad A_2=-\nabla^{-1}R_1(q_{12}),\qquad A_1=\nabla^{-1}R_2(q_{12}),\\
&A_0=\mu\sum_{m,l=1}^2 \ep^{m+1}\big[R_lR_m\big(\Re(\overline{\psi_l}\psi_m)\big)+2|\nabla|^{-1} R_l \Im( A_m \psi_m \overline{\psi_l})\big]+\frac{\mu}{2}\sum_{m=1}^2\ep^{m+1}|\psi_m|^2.
\end{split}
\end{equation}

\begin{lemma}\label{al10}
Assume $\psi_m,\psi'_m$, $m=1,2$, are as in Proposition \ref{propnon}, and define $A_m,A'_m$, $m=0,1,2$, as in \eqref{t52}. Assume 
\begin{equation*}
(G,G')\in\{(A_0,A'_0),(A_l^2,{A'_l}^2),(\partial_m A_l,\partial_mA'_l),(\psi_m \overline{\psi_l},\psi'_m\overline{\psi'_l}):m,l=1,2\}.
\end{equation*}
Then
\begin{equation} \label{fg0}
\| G \|_{L^{3/2}_t L^3_x+L^2_{x,t}}+\| G' \|_{L^{3/2}_t L^3_x+L^2_{x,t}} \lesssim a^2.
\end{equation}
and, for any $k\in\Z$,
\begin{equation} \label{fgs}
\| P_k G \|_{L^{3/2}_t L^3_x+L^2_{x,t}}+\| P_k G' \|_{L^{3/2}_t L^3_x+L^2_{x,t}} \lesssim a2^{-\sigma k}a_k(\sigma),\qquad\sigma =\sigma_0-1.
\end{equation}
Moreover,
\begin{equation} \label{fgn}
\|G'-G\|_{L^{3/2}_t L^3_x+L^2_{x,t}} \lesssim ab,
\end{equation}
and, for any $k\in\Z$,
\begin{equation} \label{fgn2}
\|P_k(G'-G)\|_{L^{3/2}_t L^3_x+L^2_{x,t}} \lesssim a2^{-\sigma k}b_k(\sigma)+b2^{-\sigma k}a_k(\sigma).
\end{equation}
\end{lemma}

\begin{proof} [Proof of Lemma \ref{al10}] We observe that the bounds \eqref{fg0} and \eqref{fgs} are implied by \eqref{fgn} and \eqref{fgn2} respectively. We prove first the bound \eqref{fgn}. Using the boundedness of the Riesz transforms on $L^p(\R^2)$, $p\in(1,\infty)$, it is clear that
\begin{equation*}
\|G'-G\|_{L^2_{x,t}} \lesssim a\sum_{m=1}^2\|\psi'_m-\psi_m\|_{L^4_{x,t}}
\end{equation*}
if $(G,G')\in\{(\partial_m A_l,\partial_mA'_l),(\psi_m \overline{\psi_l},\psi'_m\overline{\psi'_l}):m,l=1,2\}$. In addition, using the Sobolev embedding it follows that
\begin{equation}\label{al12}
\sum_{m=1}^2\|A'_m-A_m\|_{L^3_t L^6_x} \lesssim \| q_{12}-q'_{12}\|_{L^3_{t} L^{3/2}_x}\lesssim a\sum_{m=1}^2\|\psi'_m-\psi_m\|_{L^4_{x,t}}.
\end{equation}
In particular, 
\begin{equation}\label{al12.1}
\sum_{m=1}^2[\|A_m\|_{L^3_t L^6_x}+\|A'_m\|_{L^3_t L^6_x}]\lesssim a^2,
\end{equation}
thus
\begin{equation*}
\sum_{l=1}^2\|{A'_l}^2-A_l^2\|_{L^{3/2}_tL^3_x}\lesssim a^3\sum_{m=1}^2\|\psi'_m-\psi_m\|_{L^4_{x,t}},
\end{equation*}
as desired. Finally, assume that $(G,G')=(A_0,A'_0)$. For the quadratic terms in the right-hand side of \eqref{t52} we use the bound \eqref{fgn} for $(G,G')=(\psi_m \overline{\psi_l},\psi'_m\overline{\psi'_l})$, $m,l \in \{1,2\}$, which is already proved, and the boundedness of the Riesz transforms. For the remaining cubic terms we estimate using \eqref{al12}, for $m,l=1,2$,
\begin{equation*}
\begin{split}
\|\nabla^{-1} R_l \Im(A'_m \psi'_m &\overline{\psi'_l}-A_m \psi_m \overline{\psi_l}) \|_{L^{3/2}_t L^3_x} \lesssim\| A'_m \psi'_m \overline{\psi'_l}- A_m \psi_m \overline{\psi_l} \|_{L^{3/2}_t L^{6/5}_x} \\
& \lesssim \|A'_m- A_m \|_{L^3_t L^6_x} \|\psi_m \overline{\psi_l}\|_{L^3_t L^{3/2}_x}+\|A'_m\|_{L^3_tL^6_x}\|\psi'_m \overline{\psi'_l}-\psi_m \overline{\psi_l}\|_{L^3_t L^{3/2}_x}\\
&\lesssim a^3b.
\end{split}
\end{equation*}
This completes the proof  of the bounds \eqref{fgn} and \eqref{fg0}.

It remains to prove the bound \eqref{fgn2}. If $(G,G') \in \{(\partial_m A_l,\partial_mA'_l),(\psi_m \overline{\psi_l},\psi'_m\overline{\psi'_l}):m,l=1,2\}$ then, using \eqref{resaux} and \eqref{psienv2}--\eqref{psienv}, we estimate
\begin{equation*}
\|P_k(G'-G)\|_{L^2_{x,t}}\lesssim a\sum_{k'\geq k}2^{-\sigma k'} b_{k'}(\sigma)+b\sum_{k'\geq k}2^{-\sigma k'}a_{k'}(\sigma) \lesssim a 2^{-\sigma k} b_{k}(\sigma)+b2^{-\sigma k} a_{k}(\sigma),
\end{equation*}
as desired. Using again the bounds \eqref{resaux} and \eqref{psienv2}--\eqref{psienv} we estimate, for $m=1,2$ and $k\in\mathbb{Z}$,
\begin{equation}\label{al12.9}
\begin{split}
\| P_k (A'_m-A_m) \|_{L^3_t L^6_x}&\lesssim\| P_k (\psi'_1 \overline{\psi'_2}-\psi_1 \overline{\psi_2}) \|_{L^3_{t} L^{3/2}_x}\\
&\lesssim a\sum_{k' \geq k} 2^{-\sigma k'} b_{k'}(\sigma)+b\sum_{k' \geq k} 2^{-\sigma k'} a_{k'}(\sigma)\\
&\lesssim a 2^{-\sigma k} b_{k}(\sigma)+b2^{-\sigma k} a_{k}(\sigma).
\end{split}
\end{equation}
In particular,
\begin{equation}\label{al12.8}
\| P_k (A_m)\|_{L^3_t L^6_x}+\| P_k (A'_m)\|_{L^3_t L^6_x}\lesssim a 2^{-\sigma k} a_{k}(\sigma).
\end{equation}
Recall also the bounds \eqref{al12} and \eqref{al12.1}. Using again \eqref{resaux} and \eqref{psienv2}--\eqref{psienv},
\[
\| P_k({A'_m}^2-A_m^2)\|_{L^{3/2}_t L^3_x}\lesssim a^2b2^{-\sigma k} a_{k}(\sigma)+a^32^{-\sigma k}b_k(\sigma),
\]
as desired. Finally, assume $(G,G')=(A_0,A'_0)$. For the quadratic terms in the right-hand side of \eqref{t52} we use the bound \eqref{fgn2} for $(G,G')=(\psi_m \overline{\psi_l},\psi'_m\overline{\psi'_l})$, which was proved earlier. For the remaining cubic terms we use first \eqref{psienv2}--\eqref{psienv} and \eqref{resaux} to conclude that
\[
\|  P_k (\psi'_m\overline{\psi'_l}-\psi_m \overline{\psi_l}) \|_{L^3_t L^{3/2}_x} \lesssim b2^{-\sigma k} a_k(\sigma)+a2^{-\sigma k}b_k(\sigma),
\]
for $k\in\Z$ and $m,l=1,2$. In particular
\begin{equation*}
\|  P_k (\psi'_m\overline{\psi'_l})\|_{L^3_tL^{3/2}_x}+\|P_k(\psi_m \overline{\psi_l}) \|_{L^3_t L^{3/2}_x} \lesssim a2^{-\sigma k} a_k(\sigma).
\end{equation*}
Also,
\begin{equation*}
\|\psi'_m\overline{\psi'_l}-\psi_m \overline{\psi_l}\|_{L^3_t L^{3/2}_x}\lesssim ab,\qquad \|\psi'_m\overline{\psi'_l}\|_{L^3_tL^{3/2}_x}+\|\psi_m \overline{\psi_l}\|_{L^3_t L^{3/2}_x}\lesssim a^2.
\end{equation*}
Recall also the bounds \eqref{al12}, \eqref{al12.1}, \eqref{al12.9}, \eqref{al12.8} for the coefficients $A_m$, $m=1,2$,
\begin{equation*}
\begin{split}
&\|A'_m-A_m\|_{L^3_tL^6_x}\lesssim ab,\qquad \|A_m\|_{L^3_tL^6_x}+\|A'_m\|_{L^3_tL^6_x}\lesssim a^2,\\
&\| P_k (A'_m-A_m) \|_{L^3_t L^6_x}\lesssim a 2^{-\sigma k} b_{k}(\sigma)+b2^{-\sigma k} a_{k}(\sigma),\\
&\| P_k (A_m)\|_{L^3_t L^6_x}+\| P_k (A'_m)\|_{L^3_t L^6_x}\lesssim a 2^{-\sigma k} a_{k}(\sigma).
\end{split}
\end{equation*}
Combining these bounds with \eqref{resaux} leads to
\begin{equation*}
\begin{split}
\| &P_k[\nabla^{-1} R_l \Im( A'_m \psi'_m \overline{\psi'_l})-\nabla^{-1} R_l \Im( A_m \psi_m \overline{\psi_l})]\|_{L^{3/2}_t L^3_x} \\
&\lesssim \|  P_k(A'_m \psi'_m \overline{\psi'_l}- A_m \psi_m \overline{\psi_l}) \|_{L^{3/2}_t L^{6/5}_x}\\
&\lesssim a^3 2^{-\sigma k} b_k(\sigma)+a^2b2^{-\sigma k} a_k(\sigma).
\end{split}
\end{equation*}
This completes the proof of the lemma.
\end{proof}

We complete now the proof of Proposition \ref{propnon}. Recall the formula
\begin{equation*}
\mathcal{N}_m =-iA_0\psi_m+\sum_{l=1}^2\ep^{l+1}\big[\psi_l(\partial_lA_m+\partial_mA_l)+\psi_m(-\partial_lA_l+iA_l^2)\big].
\end{equation*}
The bound \eqref{nonlin3} follows from this formula and the bounds \eqref{fg0} and \eqref{fgn}. The bound \eqref{nonlin1} follows from \eqref{nonlin3} with $\psi'_m=0$.

To prove \eqref{nonlin4} we use the bounds \eqref{fg0}--\eqref{fgn2}, as well as the bounds
\begin{equation*}
\begin{split}
&\|\psi_m\|_{L^4_{x,t}}+\|\psi'_m\|_{L^4_{x,t}}\lesssim a,\qquad \|\psi'_m-\psi_m\|_{L^4_{x,t}}\lesssim b,\\
&\|P_k(\psi_m)\|_{L^4_{x,t}}+\|P_k(\psi'_m)\|_{L^4_{x,t}}\lesssim 2^{-\sigma k}a_k(\sigma),\qquad \|P_k(\psi'_m-\psi_m)\|_{L^4_{x,t}}\lesssim 2^{-\sigma k}b_k(\sigma),
\end{split}
\end{equation*}
for $m=1,2$. Using \eqref{resaux} it follows that for any $k\in\mathbb{Z}$
\begin{equation*}
\sum_{m=1}^2\|P_k(\mathcal{N}'_m-\mathcal{N}_m)\|_{L^1_tL^2_x+L^{4/3}_{x,t}}\lesssim ab2^{-\sigma k}a_k(\sigma)+a^22^{-\sigma k}b_k(\sigma).
\end{equation*}
Thus, using \eqref{psienv2}
\begin{equation*}
\begin{split}
\sum_{k\in\mathbb{Z}}2^{2\sigma k}&\|P_k(\mathcal{N}'_m-\mathcal{N}_m)\|^2_{L^1_tL^2_x+L^{4/3}_{x,t}}\\
&\lesssim a^4\sum_{m=1}^2\|\psi'_m-\psi_m\|_{X^\sigma(I)}^2+a^2b^2\sum_{m=1}^2(\|\psi'_m\|_{X^\sigma(I)}^2+\|\psi_m\|_{X^\sigma(I)}^2).
\end{split}
\end{equation*}
Using the second inequality in \eqref{ea2} and \eqref{nonlin3}
\begin{equation*}
\sum_{k\in\mathbb{Z}}\|P_k(\mathcal{N}'_m-\mathcal{N}_m)\|^2_{L^1_tL^2_x+L^{4/3}_{x,t}}\lesssim a^4\sum_{m=1}^2\|\psi'_m-\psi_m\|_{L^4_{x,t}\cap L^{12}_tL^{12/5}_x}^2. 
\end{equation*}
The bound \eqref{nonlin4} follows from the last two estimates and the first bound in \eqref{ea2}. This completes the proof of Proposition \ref{propnon}.

\subsection{Proof of Proposition \ref{regular} and Proposition \ref{wellpo}}

\begin{proof}[Proof of Proposition \ref{regular} (a)] Given $(\phi_1,\phi_2) \in H^{\sigma_0-1} \times H^{\sigma_0-1}$, it follows from \eqref{listr2} that for any $\va>0$ there is $T_\va=T_\va(\|\phi_1\|_{H^{\sigma_0-1}}+\|\phi_2\|_{H^{\sigma_0-1}}) > 0$ such that
\[
\| e^{it (\partial_1^2 + \va \partial_2^2)} \phi_i \|_{(L^4_{x,t}\cap L^{12}_tL^{12/5}_x)(\R^2\times I_\va)} \leq \va, \qquad i = 1,2, 
\]
where $I_\va=(-T_\va,T_\va)$. A standard fixed-point argument, combining the linear estimates \eqref{listr2} and the nonlinear estimates in Proposition \ref{propnon}, shows that there is $\va>0$ sufficiently small and a unique solution $(\psi_1,\psi_2)\in X^{\sigma_0-1}(I_\va)$ of the system \eqref{tot01}-\eqref{tot1}.

In addition, it is easy to combine the nonlinear estimate \eqref{nonlin3} and the linear estimate \eqref{listr} to prove the following uniqueness statement: assume $I\subseteq\R$ is an open interval and $\psi=(\psi_1,\psi_2),\psi'=(\psi'_1,\psi'_2)\in C(I:H^4)\times C(I:H^4)$ are solutions of the equations
\begin{equation*}
i\partial_t\psi_m+(\partial_1^2+\ep\partial_2^2)\psi_m=i\mathcal{N}_m,\qquad i\partial_t\psi'_m+(\partial_1^2+\ep\partial_2^2)\psi'_m=i\mathcal{N}'_m,
\end{equation*}
on $\R^2\times I$ for $m=1,2$, where $\mathcal{N}_m,\mathcal{N}'_m$ are defined as in \eqref{tot1}. If, in addition, $\psi(t_0)=\psi'(t_0)$ for some $t_0\in I$ then $\psi=\psi'$ on $I$.

The existence and uniqueness of the maximal extension $(I(\phi),\psi)$ follows by a simple argument using Zorn's lemma.
\end{proof}

\begin{proof}[Proof of Proposition \ref{regular} (b)] It is enough to prove the claim for $I_{+}(\phi)$. We do this by contradiction. Assume that $I_+(\phi)=[0,T_+)$ is bounded and
\begin{equation}\label{nj1}
\sum_{m=1}^2\|f_m\|_{L^2_x}+\sum_{m=1}^2 \| \psi_m \|_{L^4_{x,t}(\R^2 \times I_+(\phi))}\leq A\in[1,\infty).
\end{equation}
It follows from \eqref{tot01} and \eqref{tot1} that
\begin{equation*}
\partial_t(\psi_m\overline{\psi_m})=i\overline{\psi_m}(\partial_1^2+\ep\partial_2^2)\psi_m-i\psi_m(\partial_1^2+\ep\partial_2^2)\overline{\psi_m}+\mathcal{N}_m\overline{\psi_m}+\overline{\mathcal{N}_m}\psi_m.
\end{equation*}
Thus
\begin{equation*}
\begin{split}
\frac{d}{dt}&\int_{\R^2}|\psi_m|^2\,dx=\int_{\R^2}(\mathcal{N}_m\overline{\psi_m}+\overline{\mathcal{N}_m}\psi_m)\,dx\\
&=\int_{\R^2}\big[-2\psi_m\overline{\psi_m}(\partial_1A_1+\ep\partial_2A_2)+2\sum_{l=1}^2\ep^{l+1}(\partial_lA_m+\partial_mA_l)\Re(\psi_m\overline{\psi_l})\big]\,dx.
\end{split}
\end{equation*}
Using \eqref{nj1} and the definition of $A_1,A_2$, it follows that
\begin{equation}\label{nj2}
\sum_{m=1}^2 \| \psi_m \|_{(L^\infty_tL^2_x\cap L^4_{x,t})(\R^2 \times I_+(\phi))}\lesssim A^2.
\end{equation}
Thus, for any $\va>0$ there is $M=M(\va,A)$ and a partition $I_+(\phi)=I_1\cup\ldots\cup I_M$, $I_l=[T_{l-1},T_l)$, $T_0=0$, $T_M=T_+$, with the property that
\begin{equation}\label{nj3}
\sum_{m=1}^2 \| \psi_m \|_{(L^{12}_tL^{12/5}_x\cap L^4_{x,t})(\R^2 \times I_l)}\leq \va,\qquad l=1,\ldots,M.
\end{equation}
The nonlinear bound \eqref{nonlin2} and the second linear bound in \eqref{listr2} show that, for any $l=1,\ldots,M$ 
\begin{equation*}
\sup_{t\in I_l}\sum_{m=1}^2 \| \psi_m(t) \|_{H^{\sigma_0-1}}\lesssim \sum_{m=1}^2 \| \psi_m(T_{l-1})\|_{H^{\sigma_0-1}}.
\end{equation*}
As a consequence
\begin{equation}\label{nj4}
\sup_{t\in I_+(\phi)}\sum_{m=1}^2 \| \psi_m(t) \|_{H^{\sigma_0-1}}\lesssim_A \sum_{m=1}^2 \|f_m\|_{H^{\sigma_0-1}}.
\end{equation}
Moreover, the functions $t\to(\psi_1(t),\psi_2(t))$ converge in $H^{\sigma_0-1}\times H^{\sigma_0-1}$ as $t\to T_+$, in contradiction with the maximality of $I_+(\phi)$.
\end{proof}

\begin{proof}[Proof of Proposition \ref{regular} (c)] We need to prove that $\DD_m$ are covariant derivatives in the sense of \eqref{ea9} assuming that $\psi_m$, $m=1,2$ and $A_m$, $m=0,1,2$ satisfy the identities in Proposition \ref{regular} (a), $\DD_1 \psi_2 = \DD_2 \psi_1$ at $t=0$, and 
\begin{equation*}
\psi_0=i(\DD_1\psi_1+\ep\DD_2\psi_2)+2A_1\psi_1+2\ep A_2\psi_2.
\end{equation*}
From \eqref{tot1} it follows that
\begin{equation} \label{cd1}
[\DD_1,\DD_2] = i(\partial_1 A_2 - \partial_2 A_1) =i\mu\Im(\psi_1\overline{\psi_2})
\end{equation}
and this is the only covariant property which can be derived directly.

We define, for $m=1,2$,
\[
F=\DD_1 \psi_2 - \DD_2 \psi_1, \qquad H_m = \DD_m \psi_0 - \DD_0 \psi_m, \qquad G_m = \partial_mA_0-\partial_0A_m- q_{m0},
\]
where $q_{m0}=\mu\Im(\overline{\psi_0}\psi_m)$. The idea is to write an equation for the evolution in time for $F$ which allows us, under suitable conditions on the coefficients on some time interval $I \ni 0$, to prove that $F(0)=0$ implies $F(t)=0$ for all $t \in I$. Such a computation involves $H_m, G_m, m=1,2$, therefore we start by connecting these expressions to $F$. 

The Schr\"odinger equation in \eqref{regular} was derived starting from exploiting the fact that $H_m=0$, see \eqref{covaux}. We can redo the computations in \eqref{covaux}-\eqref{iscoord} assuming only \eqref{cd1} and the identities in Proposition \ref{regular} (a); the result is
\begin{equation*}
H_1 = i \ep(\partial_2-iA_2)F, \qquad  H_2 = - i(\partial_1-iA_1)F.
\end{equation*}

Next we want to relate $G_m, m=1,2$ to $F$. Undoing the computation that derived \eqref{a0f} from \eqref{da0}, and taking into account that \eqref{cd1} holds, and $\DD_1 \psi_2 - \DD_2 \psi_1 =F$, we derive first
\begin{equation*}
\Delta A_0 = \mu\sum_{l=1}^2 \partial_l \Im(\psi_l \overline{\psi_0}) - \mu \partial_2 \Re{(\psi_1 \overline{F})} + \mu \va \partial_1 \Re{(\psi_2 \overline{F})}.
\end{equation*}
Then we continue with
\begin{equation*}
\begin{split}
&\Delta G_1 = \Delta (\partial_1 A_0 - \partial_0A_1 - \mu \Im{(\psi_1 \overline{\psi_0})}) \\
&= \mu  \partial_1 \partial_2 \Im(\psi_2 \overline{\psi_0}) - \mu \partial_1 \partial_2 \Re{(\psi_1 \overline{F})} + \mu \ep \partial_1^2 \Re{(\psi_2 \overline{F})} +  \mu \partial_2 \partial_0 \Im{(\psi_1 \overline{\psi_2})} - \mu \partial_2^2 \Im{(\psi_1 \overline{\psi_0})} \\
&= - \mu \partial_1 \partial_2 \Re{(\psi_1 \overline{F})} + \mu \ep \partial_1^2 \Re{(\psi_2 \overline{F})} + \mu \partial_2 [\partial_1\Im(\psi_2 \overline{\psi_0})+ \partial_0  \Im{(\psi_1 \overline{\psi_2})} - \partial_2 \Im{(\psi_1 \overline{\psi_0})}] \\
&= -\mu \partial_1 \partial_2 \Re{(\psi_1 \overline{F})} + \mu \ep \partial_1^2 \Re{(\psi_2 \overline{F})} + \mu \partial_2  I.
\end{split}
\end{equation*}
Based on the formulas derived above for $H_1,H_2$, we compute $I$ separately,
\begin{equation*}
\begin{split}
I&= \Im(\DD_1 \psi_2 \overline{\psi_0}+\psi_2 \overline{\DD_1 \psi_0}+\DD_0 \psi_1 \overline{\psi_2}+\psi_1 \overline{\DD_0 \psi_2} -\DD_2 \psi_1 \overline{\psi_0} -\psi_1 \overline{\DD_2 \psi_0})\\
&= \Im{(F \overline{\psi_0})} - \Im{(\overline{\psi_2} H_1)} + \Im{(\overline{\psi_1} H_2)} \\
&= -\partial_1 \Re{(F \cdot \overline{\psi_1})} - \ep \partial_2 \Re{(F \cdot \overline{\psi_2})}.
\end{split}
\end{equation*}
Thus
\[
\Delta G_1 = \mu \ep (\partial_1^2 - \partial_2^2) \Re{(F \cdot \overline{\psi_2})}-2\mu\partial_1\partial_2\Re{(F \cdot \overline{\psi_1})},
\]
which gives
\[
G_1 = -\mu \ep(R_1^2 - R_2^2)( \Re{(F\overline{\psi_2})})+2\mu R_1R_2(\Re{(F \overline{\psi_1})}).
\]
In a similar manner one obtains
\[
G_2 =  -\mu(R_1^2 - R_2^2)( \Re{(F\overline{\psi_1})})-2\mu\ep R_1R_2(\Re{(F \overline{\psi_2})}).
\]
In particular, for any $t\in I$,
\begin{equation}\label{ea100}
\|G_1\|_{L^2_x}+\|G_2\|_{L^2_x}\lesssim \|F\|_{L^2_x}(\|\psi_1\|_{L^\infty_x}+\|\psi_2\|_{L^\infty_x}).
\end{equation}

We derive now an evolution equation for $F$. We begin with rewriting the evolution equation for each $\psi_m$ as follows
\begin{equation*}
\DD_0 \psi_m = \DD_m \psi_0 - H_m = \DD_m (i\DD_1  \psi_1+2A_1\psi_1 +i\ep\DD_2 \psi_2+2\ep A_2\psi_2) - H_m.
\end{equation*}
From this it follows
\begin{equation*}
\begin{split}
&\DD_0 F +iG_1 \psi_2-iG_2 \psi_1\\
&= \DD_0 (\DD_1 \psi_2 - \DD_2 \psi_1)+iG_1 \psi_2-iG_2 \psi_1 \\
&= \DD_1 \DD_0 \psi_2 - \DD_2 \DD_0 \psi_1 + i q_{01} \psi_2 - i q_{02} \psi_1 \\
&=\DD_1 \DD_2 \psi_0 - \DD_1 H_2 - \DD_2 \DD_1 \psi_0 + \DD_2 H_1 + i q_{01} \psi_2 - i q_{02} \psi_1 \\
&= i(\partial_1^2+\ep\partial_2^2)F  + (\partial_1A_1+iA_1^2+\ep\partial_2A_2+i\ep A_2^2)F+i q_{12} \psi_0 + i q_{01} \psi_2 - i q_{02} \psi_1 \\
&= i(\partial_1^2+\ep\partial_2^2)F  + (\partial_1A_1+iA_1^2+\ep\partial_2A_2+i\ep A_2^2)F.
\end{split}
\end{equation*}
We multiply the equation by $\overline{F}$, integrate over $\R^2$, and take the real part to obtain
\begin{equation*}
\begin{split}
\frac{1}{2}\frac{d}{dt} \int |F|^2 dx &= \int \Re[(-iG_1 \psi_2 +iG_2 \psi_1)\overline{F}]+(\partial_1A_1+\ep\partial_2A_2)F\overline{F}\,dx \\
& \lesssim ( \| \psi_1 \|^2_{L^\infty} +  \| \psi_2 \|^2_{L^\infty}+ \|\partial_1A_1 \|_{L^\infty} +  \|\partial_2A_2\|_{L^\infty}) \|F\|_{L^2}^2,
\end{split}
\end{equation*}
where in the last line we have used \eqref{ea100}. Since $F(0)=0$, it follows that $F(t)=0$ for all $t \in I$. As a consequence $H_m=G_m=0, m=1,2$, hence the full covariant calculus is preserved.
\end{proof}

\begin{proof}[Proof of Proposition \ref{wellpo}] As in the proof of Proposition \ref{regular} (b), see the proof of \eqref{nj2}, we have
\begin{equation*}
\sum_{m=1}^2 \| \psi_m \|_{(L^\infty_tL^2_x\cap L^4_{x,t})(\R^2 \times J)}\lesssim 1+N_{\psi,J}^2.
\end{equation*}
As in the proof  of Proposition \ref{regular} (b), for any $\va>0$ we partition the interval $J$ into $M=M(\va,N_{\psi,J}^2)$ closed subintervals $J_1,\ldots, J_M$ such that
\begin{equation*}
\sum_{m=1}^2 \| \psi_m \|_{(L^{12}_tL^{12/5}_x\cap L^4_{x,t})(\R^2 \times J_l)}\leq \va,\qquad l=1,\ldots,M.
\end{equation*}
The proposition follows by applying Lemma \ref{onestep} below on every subinterval $J_l$.
\end{proof}

\begin{lemma}\label{onestep}
Assume $J=[t_1,t_2]$ is a compact interval and $\psi_m\in C(J:H^{\sigma_0-1})$, $m=1,2$, are solutions of the equations
\begin{equation*}
i\partial_t\psi_m+(\partial_1^2+\ep\partial_2^2)\psi_m=i\mathcal{N}_m,
\end{equation*}
with $\mathcal{N}_m$ defined as  in \eqref{tot1}. Assume, in addition, that
\begin{equation*}
\sum_{m=1}^2\|\psi_m\|_{(L^{12}_tL^{12/5}_x\cap L^4_{x,t})(\R^2 \times J)}\leq \va_0,
\end{equation*}
for  some $\va_0$ sufficiently small.  Then  there is $\delta_0$ sufficiently small with the following property: if $(\phi'_1,\phi'_2)\in H^{\sigma_0-1}\times H^{\sigma_0-1}$ and
\begin{equation*}
\sum_{m=1}^2\|\phi'_m-\psi_m(t_1)\|_{L^2(\R^2)}=\delta\leq\delta_0
\end{equation*}
then there is a solution $(\psi'_1,\psi'_2)\in C(J:H^{\sigma_0-1})\times C(J:H^{\sigma_0-1})$ of the system
\begin{equation*}
i\partial_t\psi'_m+(\partial_1^2+\ep\partial_2^2)\psi'_m=i\mathcal{N}'_m,\quad \psi'_m(t_1)=\phi'_m,
\end{equation*}
with $\mathcal{N}'_m$ defined as in \eqref{tot1}, and
\begin{equation}\label{nj6}
\sum_{m=1}^2\|\mathcal{N}'_m-\mathcal{N}_m\|_{(L^1_tL^2_x+L^{4/3}_{x,t})(\R^2\times J)}\leq\delta.
\end{equation}
\end{lemma}

\begin{proof}[Proof of Lemma \ref{onestep}] We may assume $t_1=0$. From Proposition \ref{regular} (a), there is $t_2' > 0$ such that $J'=[0,t_2'] \subseteq I(\phi')$. Choose $t_2'\in(0,t_2]\cap I(\phi')$ with the property
\[
\sum_{m=1}^2\| \psi'_m\|_{(L^{12}_tL^{12/5}_x\cap L^4_{x,t})(\R^2 \times J')}\leq 2\va_0.
\]
By applying \eqref{listr2} and \eqref{nonlin3} it follows that
\begin{equation*}
\begin{split}
\| \psi_m -&\psi_m' \|_{(L^{12}_tL^{12/5}_x\cap L^4_{x,t})(\R^2 \times J')}\\
&\leq C \| \psi_m(0) -\psi'_m(0)\|_{L^2_x} +C\va_0^2 \| \psi_m -\psi'_m\|_{(L^{12}_tL^{12/5}_x\cap L^4_{x,t})(\R^2 \times J')}.
\end{split}
\end{equation*}
Thus, if $\va_0$ is sufficiently  small,
\[
\| \psi_m -\psi_m' \|_{(L^{12}_tL^{12/5}_x\cap L^4_{x,t})(\R^2 \times J')}\leq 2C\delta.
\]
As a consequence, if  $\delta_0$ is sufficiently small, $\sum_{m=1}^2\| \psi_m' \|_{(L^{12}_tL^{12/5}_x\cap L^4_{x,t})(\R^2 \times J')}\leq 3\va_0/2$. The function $H:I_+(\phi')\to[0,\infty)$, 
\begin{equation*}
H(t)=\sum_{m=1}^2\| \psi'_m\|_{(L^{12}_tL^{12/5}_x\cap L^4_{x,t})(\R^2 \times[0,t])},
\end{equation*}
is continuous, $H(0)=0$, and, as proved above, 
\begin{equation*}
\text{ if }t\in I(\phi')\cap [0,t_2]\text{ and }H(t)\leq 2\va_0\text{ then }H(t)\leq 3\va_0/2.
\end{equation*}
It follows that $H(t)\leq 3\va_0/2$ for any $t\in I(\phi')\cap [0,t_2]$. Using Proposition \ref{regular} (b), it follows that $[0,t_2]\subseteq I(\phi')$. In addition, using \eqref{nonlin3} and \eqref{listr2},
\begin{equation*}
\begin{split}
\sum_{m=1}^2&\|\mathcal{N}'_m-\mathcal{N}_m\|_{(L^1_tL^2_x+L^{4/3}_{x,t})(\R^2\times J)}\leq C\va_0^2\sum_{m=1}^2\|\psi'_m-\psi_m\|_{(L^{12}_tL^{12/5}_x\cap L^4_{x,t})(\R^2\times J)}\\
&\leq C\va_0^2\big[C\sum_{m=1}^2\|f'_m-\psi_m(0)\|_{L^2_x}+C\sum_{m=1}^2\|\mathcal{N}'_m-\mathcal{N}_m\|_{(L^1_tL^2_x+L^{4/3}_{x,t})(\R^2\times J)}\big].
\end{split}
\end{equation*}
The bound  \eqref{nj6} follows, if $\va_0$ is sufficiently small.
\end{proof}

\section{Proof of Theorem \ref{thm1}}\label{proofth1}

In this section we prove Theorem \ref{thm1}. Given data $f\in\widetilde{H}^{\sigma_0}$ as in Theorem \ref{thm1} we construct first a suitable Coulomb frame $(v,w)$ in $T_fS_\mu$ and the fields $\phi_1, \phi_2\in H^{\sigma_0-1}$. Then we construct the maximal solution $(I(\phi),\psi)$ of the modified spin system, using Proposition \ref{regular}. Finally, we construct the maximal solution $s$ on the interval $I(f)=I(\phi)$, by integrating the fields $\psi_m$.

We prove now the uniqueness of the maximal solution $(I(f),s)$. For this it suffices  to prove the following uniqueness statement:

\begin{proposition}\label{uniquesta} Assume $I\subseteq\R$ is an open interval, $t_0\in I$, and $s,s'\in C(I:\widetilde{H}^{\sigma_0})$ are solutions of the equations
\begin{equation*}
\begin{split}
&\partial_t s= s \times_\mu (s_{11}+\ep s_{22})+s_1\ze_2-\ep s_2\ze_1,\qquad\ze_m=-R_m\nabla^{-1}[2 \mu s\cdot_\mu(s_1 \times_\mu s_2)],\\
&\partial_t s'= s' \times_\mu (s'_{11}+\ep s'_{22})+s'_1\ze'_2-\ep s'_2\ze'_1,\qquad\ze'_m=-R_m\nabla^{-1}[2 \mu s'\cdot_\mu(s'_1 \times_\mu s'_2)].
\end{split}
\end{equation*}
Assume also that $s(t_0)=s'(t_0)$. Then $s=s'$ on $\R^2\times  I$. 
\end{proposition}

\begin{proof}[Proof of Proposition \ref{uniquesta}] We use first Proposition \ref{vwoni} and Proposition \ref{summary} to construct Coulomb frames $(v,w)$ and $(v',w')$, and fields $\psi_1,\psi_2,\psi'_1,\psi'_2\in C(I:H^4)$, which solve the evolution equations
\begin{equation*}
i\partial_t\psi_m+(\partial_1^2+\ep\partial_2^2)\psi_m=i\mathcal{N}_m,\quad i\partial_t\psi'_m+(\partial_1^2+\ep\partial_2^2)\psi'_m=i\mathcal{N}'_m,
\end{equation*} 
where $\mathcal{N}_m,\mathcal{N}'_m$ are defined as in \eqref{bas3}. Since $s(t_0)=s'(t_0)$, we may assume that the frames $(v,w)$ and $(v',w')$ agree at time $t=t_0$ (by choosing $v(0,0,t_0)=v'(0,0,t_0)$). Thus $\psi_m(t_0)=\psi'_m(t_0)$, $m=1,2$. The uniqueness statement in the proof of Proposition \ref{regular} (a) shows that $\psi_m=\psi'_m$ on $\R^2\times I$. The formulas \eqref{bas4} and \eqref{bas2} show that $\psi_0=\psi'_0$ and $A_m=A'_m$ on $\R^2\times I$, $m=0,1,2$. Finally, the linear systems \eqref{basis} show that $s=s'$, $v=v'$, $w=w'$ in $\R^2\times I$, as desired.
\end{proof}

We construct now the initial-data fields $\phi_1,\phi_2$.

\begin{proposition}\label{vwonetime}
Assume $f\in \widetilde{H}^{\sigma_0}$ and $Q\in\R^3$, $Q\cdot_\mu f(0,0)=0$, $Q\cdot_\mu Q=1$. Then there are unique $C^3$ functions $v,w:\R^2\to\R^3$,
\begin{equation}\label{nj16}
v\cdot_\mu f=0,\quad v\cdot_\mu v=1,\quad w=f\times_\mu v,\quad v(0,0)=Q,
\end{equation}
with the following property: if we define
\begin{equation}\label{lo1}
\phi_m=v\cdot_\mu\partial_m f+iw\cdot_\mu\partial_m f,\quad A_m=w\cdot_\mu\partial_m v
\end{equation}
then $\phi_m\in H^{\sigma_0-1}$, $m=1,2$,
\begin{equation}\label{nj17}
A_2=-\nabla^{-1}R_1(q_{12}),\quad A_1=\nabla^{-1}R_2(q_{12}),\quad q_{12}=\mu\Im(\phi_1\overline{\phi_2}),
\end{equation}
\begin{equation}\label{nj18}
(\partial_1+iA_1)\phi_2=(\partial_2+iA_2)\phi_1.
\end{equation}
and, for $m=1,2$,
\begin{equation}\label{nj30}
\begin{cases}
&\partial_m f=v\Re(\phi_m)+w\Im(\phi_m),\\
&\partial_m v=-f\mu\Re(\phi_m)+wA_m,\\
&\partial_m w=-f\mu \Im(\phi_m)-vA_m.
\end{cases}
\end{equation}
\end{proposition}

\begin{proof}[Proof of Proposition \ref{vwonetime}]
The existence and uniqueness of the frame $v,w$ is a consequence of Proposition \ref{vwoni} (applied  to the function $s:\R^2\times (-1,1)$, $s(x,t)=f(x)$). The identities \eqref{nj16}-\eqref{nj30} are derived in section \ref{gauge}.

It remains to prove that $\phi_1,\phi_2\in H^{\sigma_0-1}$. For this we use a simple elliptic bootstrap argument based on the system \eqref{nj30} and the identities
\begin{equation}\label{nj30.5}
\begin{cases}
&\phi_m=v\cdot_\mu\partial_m f+iw\cdot_\mu\partial_m f,\\
&A_2=-\nabla^{-1}R_1(q_{12}),\qquad A_1=\nabla^{-1}R_2(q_{12}),\qquad q_{12}=\mu\Im(\phi_1\overline{\phi_2}).
\end{cases}
\end{equation}
Let $Q_0={}^t(1,0,0)\in S_\mu$ and
\begin{equation*}
\|f\|_{\widetilde{H}^{\sigma_0}}=d_{\sigma_0}(f,Q_0)=\|f-Q_0\|_{L^\infty}+\sum_{m=1}^2\|\partial_mf\|_{H^{\sigma_0-1}}.
\end{equation*}
By construction, $f,v,w,\phi_m,A_m$ are continuous functions on $\R^2$ and
\begin{equation}\label{nj31}
\|f\|_{L^\infty}+\|v\|_{L^\infty}+\|w\|_{L^\infty}+\sum_{m=1}^2(\|\phi_m\|_{L^2\cap L^\infty}+\|A_m\|_{L^4\cap L^\infty})\leq C(\|f\|_{\widetilde{H}^{\sigma_0}}).
\end{equation}
For $\sigma=0,1,\ldots$ and $p\in[1,\infty]$ let 
\begin{equation*}
\|\partial^\sigma g\|_{L^p}=\sum_{\sigma_1+\sigma_2=\sigma}\|\partial_1^{\sigma_1}\partial_2^{\sigma_2}g\|_{L^p}.
\end{equation*}
It follows from \eqref{nj30} and \eqref{nj31} that
\begin{equation*}
\|\partial^1f\|_{L^4\cap L^\infty}+\|\partial^1v\|_{L^4\cap L^\infty}+\|\partial^1w\|_{L^4\cap L^\infty}\leq C(\|f\|_{\widetilde{H}^{\sigma_0}}).
\end{equation*}
Using \eqref{nj30.5} it follows that $\|\partial^1\phi_m\|_{L^2\cap L^\infty}\leq C(\|f\|_{\widetilde{H}^{\sigma_0}})$ and then $\|\partial^1A_m\|_{L^2\cap L^\infty}\leq C(\|f\|_{\widetilde{H}^{\sigma_0}})$. Thus
\begin{equation*}
\begin{split}
\|\partial^1f\|_{L^4\cap L^{p_1}}&+\|\partial^1v\|_{L^4\cap L^{p_1}}+\|\partial^1w\|_{L^4\cap L^{p_1}}\\
&+\sum_{m=1}^2(\|\partial^1\phi_m\|_{L^2\cap L^{p_1}}+\|\partial^1A_m\|_{L^2\cap L^{p_1}})\leq C(\|f\|_{\widetilde{H}^{\sigma_0}}),
\end{split}
\end{equation*}
where $p_1=p_1({\sigma_0})$ is sufficiently large. A simple inductive argument then shows that
\begin{equation}\label{nj33}
\begin{split}
\|\partial^\sigma f\|_{L^4\cap L^{p_{\sigma}}}&+\|\partial^\sigma v\|_{L^4\cap L^{p_\sigma}}+\|\partial^\sigma w\|_{L^4\cap L^{p_\sigma}}\\
&+\sum_{m=1}^2(\|\partial^\sigma\phi_m\|_{L^2\cap L^{p_\sigma}}+\|\partial^\sigma A_m\|_{L^2\cap L^{p_\sigma}})\leq C(\|f\|_{\widetilde{H}^{\sigma_0}}),
\end{split}
\end{equation}
for $\sigma=1,\ldots,\sigma_0-2$ and $p_{\sigma}=p_1/2^{\sigma-1}$. We apply \eqref{nj30.5} one more time to conclude that $\|\partial^{\sigma_0-1} v\|_{L^4}+\|\partial^{\sigma_0-1} w\|_{L^4}\leq C(\|f\|_{\widetilde{H}^{\sigma_0}})$. Finally, we use the first identity in \eqref{nj30.5} to conclude that $\|\partial^{\sigma_0-1} \phi_m\|_{L^2}\leq C(\|f\|_{\widetilde{H}^{\sigma_0}})$, $m=1,2$, as desired. 
\end{proof}

\subsection{Construction of the maximal solution $(I(f),s)$}\label{psitos} 
In this subsection we construct a maximal solution $s$ of the initial-value problem \eqref{in1}. Given data $f\in\widetilde{H}^{\sigma_0}$, we construct a frame $(v,w)$ and the fields $\phi_1,\phi_2\in H^{\sigma_0-1}$ as in Proposition \ref{vwonetime}. Then we construct the maximal solution $\psi=(\psi_1,\psi_2)\in C(I(\phi):H^{\sigma_0-1})\times  C(I(\phi):H^{\sigma_0-1})$ as in Proposition \ref{regular} (a), 
\begin{equation}\label{ny1}
i \partial_t \psi_m +(\partial_1^2+\ep\partial_2^2) \psi_m = i\mathcal{N}_m,\qquad\psi_m(0)=\phi_m,
\end{equation}
where
\begin{equation}\label{ny2}
\begin{split}
&\mathcal{N}_m =-iA_0\psi_m+\sum_{l=1}^2\ep^{l+1}\big[\psi_l(\partial_lA_m+\partial_mA_l)+\psi_m(-\partial_lA_l+iA_l^2)\big],\\
&q_{12}=\mu\Im(\psi_1\overline{\psi_2}),\qquad A_2=-\nabla^{-1}R_1(q_{12}),\qquad A_1=\nabla^{-1}R_2(q_{12}),\\
&A_0=\mu\sum_{m,l=1}^2 \ep^{m+1}\big[R_lR_m\big(\Re(\overline{\psi_l}\psi_m)\big)+2|\nabla|^{-1} R_l \Im( A_m \psi_m \overline{\psi_l})\big]+\frac{\mu}{2}\sum_{m=1}^2\ep^{m+1}|\psi_m|^2.
\end{split}
\end{equation}
In view of Proposition \ref{regular} (c) and  \eqref{nj18}, the identities
\begin{equation}\label{ny3}
\DD_l\psi_m=\DD_m\psi_l,\qquad \partial_lA_m-\partial_mA_l=\mu\Im(\psi_l\overline{\psi_m}),\qquad m,l=0,1,2,
\end{equation} 
hold in $\R^2\times I(\phi)$, where $\DD_m=\partial_m+iA_m$ and
\begin{equation}\label{ny3.5}
\psi_0=i(\DD_1\psi_1+\ep\DD_2\psi_2)+2A_1\psi_1+2\ep A_2\psi_2=\sum_{m=1}^2\ep^{m+1}(i\partial_m\psi_m+A_m\psi_m).
\end{equation}

At time $t=0$, the functions $f,v,w,\phi_m,A_m$ satisfy the identities
\begin{equation}\label{ny4}
\begin{cases}
&\partial_m f=v\Re(\phi_m)+w\Im(\phi_m),\\
&\partial_m v=-f\mu\Re(\phi_m)+wA_m,\\
&\partial_m w=-f\mu \Im(\phi_m)-vA_m,
\end{cases}
\end{equation}
for $m=1,2$ (compare with \eqref{nj30}), and
\begin{equation}\label{ny5}
v\cdot_\mu f=w\cdot_\mu f=v\cdot_\mu w=0,\quad v\cdot_\mu v=w\cdot_\mu w=\mu f\cdot_\mu f=1.
\end{equation}

We define the functions $C^1$ functions $s,v,w:\R^2\times I(\phi)\to\R^2$ as the solutions of the linear homogeneous ordinary differential equations
\begin{equation}\label{ny6}
\begin{cases}
&\partial_0 s=v\Re(\psi_0)+w\Im(\psi_0),\\
&\partial_0 v=-s\mu\Re(\psi_0)+wA_0,\\
&\partial_0 w=-s\mu \Im(\psi_0)-vA_0,
\end{cases}
\end{equation}
where $s(,.0)=f$, $v(.,0)$ and $w(.,0)$ are defined as before (compare with \eqref{basis}). 

We show first that the identities in \eqref{ny5} continue to hold in $\R^2\times I(\phi)$. For this we compute, using the definition \eqref{ny6},
\begin{equation*}
\begin{split}
\partial_0(v\cdot_\mu s)&=(-s\mu\Re(\psi_0)+wA_0)\cdot_\mu s+(v\Re(\psi_0)+w\Im(\psi_0))\cdot_\mu v\\
&=A_0(w\cdot_\mu s)+\Im(\psi_0)(v\cdot_\mu w)+\Re(\psi_0)(v\cdot_\mu v-1)-\Re(\psi_0)(\mu s\cdot_\mu s-1).
\end{split}
\end{equation*}
Similarly, we compute
\begin{equation*}
\begin{split}
&\partial_0(w\cdot_\mu s)=-A_0(v\cdot_\mu s)+\Re(\psi_0)(v\cdot_\mu w)+\Im(\psi_0)(w\cdot_\mu w-1)-\Im(\psi_0)(\mu s\cdot_\mu s-1),\\
&\partial_0(v\cdot_\mu w)=-\mu\Re(\psi_0)(w\cdot_\mu s)-\mu\Im(\psi_0)(v\cdot_\mu s)+A_0(w\cdot_\mu w-1)-A_0(v\cdot_\mu v-1),\\
&\partial_0(\mu s\cdot_\mu s-1)=2\mu\Re(\psi_0)(v\cdot_\mu s)+2\mu\Im(\psi_0)(w\cdot_\mu s),\\
&\partial_0(v\cdot_\mu v-1)=-2\mu\Re(\psi_0)(v\cdot_\mu s)+2A_0(w\cdot_\mu s),\\
&\partial_0(w\cdot_\mu w-1)=-2\mu\Im(\psi_0)(w\cdot_\mu s)-2A_0(v\cdot_\mu w).
\end{split}
\end{equation*}
In view of \eqref{ny5}, it follows that
\begin{equation}\label{ny7}
v\cdot_\mu s=w\cdot_\mu s=v\cdot_\mu w=\mu s\cdot_\mu s-1=v\cdot_\mu v-1=w\cdot_\mu w-1=0\qquad\text{ on }\R^2\times I(\phi).
\end{equation}
In addition, since $s\times_\mu v=w$, $w\times_\mu s=v$, and $v\times_\mu w=\mu s$ at $t=0$, we have, by continuity,
\begin{equation}\label{ny8}
s\times_\mu v=w,\quad v\times_\mu w=\mu s,\quad w\times_\mu s=v\qquad\text{ on }\R^2\times I(\phi).
\end{equation}

We prove now that the identities \eqref{ny4} continue to hold in $\R^2\times I(\phi)$, for $m=1,2$. For $m=1,2$ let $X_m=\partial_m s-v\Re(\psi_m)-w\Im(\psi_m)$, $Y_m=\partial_m v+s\mu\Re(\psi_m)-wA_m$, $Z_m=\partial_m w+s\mu \Im(\psi_m)+vA_m$. Using the definition \eqref{ny6} we compute
\begin{equation*}
\begin{split}
\partial_0(X_m)&=\partial_m[v\Re(\psi_0)+w\Im(\psi_0)]-\partial_0[v\Re(\psi_m)+w\Im(\psi_m)]\\
&=v[\partial_m\Re(\psi_0)-\partial_0\Re(\psi_m)]+w[\partial_m\Im(\psi_0)-\partial_0\Im(\psi_m)]\\
&+(\partial_mv)\Re(\psi_0)+(\partial_mw)\Im(\psi_0)-(\partial_0v)\Re(\psi_m)-(\partial_0w)\Im(\psi_m).
\end{split}
\end{equation*}
Using $\partial_m v=Y_m+wA_m-s\mu\Re(\psi_m)$ and $\partial_mw=Z_m-vA_m-s\mu\Im(\psi_m)$, and the identities \eqref{ny6}, this  becomes
\begin{equation*}
\begin{split}
&Y_m\Re(\psi_0)+Z_m\Im(\psi_0)+v[\partial_m\Re(\psi_0)-\partial_0\Re(\psi_m)-A_m\Im(\psi_0)+A_0\Im(\psi_m)]\\
&+w[\partial_m\Im(\psi_0)-\partial_0\Im(\psi_m)+A_m\Re(\psi_0)-A_0\Re(\psi_m)]\\
&=Y_m\Re(\psi_0)+Z_m\Im(\psi_0)+v\Re(\DD_m\psi_0-\DD_0\psi_m)+w\Im(\DD_m\psi_0-\DD_0\psi_m).
\end{split}
\end{equation*}
Using  the identities \eqref{ny3}, it follows that
\begin{equation*}
\partial_0(X_m)=Y_m\Re(\psi_0)+Z_m\Im(\psi_0).
\end{equation*}
Similar computations give
\begin{equation*}
\partial_0(Y_m)=-X_m\mu\Re(\psi_0)+Z_mA_0,\qquad \partial_0(Z_m)=-X_m\mu\Im(\psi_0)-Y_mA_0.
\end{equation*}
Since $X_m,Y_m,Z_m$ vanish at $t=0$, we conclude that the identities 
\begin{equation}\label{ny10}
\begin{cases}
&\partial_m s=v\Re(\psi_m)+w\Im(\psi_m),\\
&\partial_m v=-s\mu\Re(\psi_m)+wA_m,\\
&\partial_m w=-s\mu \Im(\psi_m)-vA_m,
\end{cases}
\end{equation}
hold in $\R^2\times I(\phi)$, for $m=0,1,2$.

Using \eqref{ny10}, \eqref{ny7}, \eqref{ny8}, we derive
\begin{equation*}
2\mu s\cdot_\mu(s_1\times_\mu s_2)=2\mu[\Re(\psi_1)\Im(\psi_2)-\Im(\psi_1)\Re(\psi_2)]=-2q_{12},
\end{equation*}
thus, using \eqref{sense} and \eqref{ny2},
\begin{equation}\label{ny11}
\ze_1=-2A_2,\qquad\ze_2=2A_1.
\end{equation}
Then, using \eqref{ny10} and \eqref{ny8},
\begin{equation*}
\begin{split}
&s\times_\mu(s_{11}+\ep s_{22})+s_1\ze_2-\ep s_2\ze_1\\
&=\sum_{m=1}^2 \ep^{m+1}\big[s\times_\mu\partial_m(v\Re(\psi_m)+w\Im(\psi_m))+2(v\Re(\psi_m)+w\Im(\psi_m))A_m\big]\\
&=\sum_{m=1}^2 \ep^{m+1}\big[w(\partial_m\Re(\psi_m)+A_m\Im(\psi_m))+v(-\partial_m\Im(\psi_m)+A_m\Re(\psi_m)\big].
\end{split}
\end{equation*}
Using \eqref{ny10} and the definition \eqref{ny3.5}
\begin{equation*}
\begin{split}
\partial_0s&=w\Im(\psi_0)+v\Re(\psi_0)\\
&=w\sum_{m=1}^2 \ep^{m+1}(\partial_m\Re(\psi_m)+A_m\Im(\psi_m))+v\sum_{m=1}^2 \ep^{m+1}(-\partial_m\Im(\psi_m)+A_m\Re(\psi_m)).
\end{split}
\end{equation*}
Therefore $s$ is a solution of the initial-value problem \eqref{in1}, as desired.

We show now that $s\in C(I(\phi):\widetilde{H}^{\sigma_0})$. The definition \eqref{ny6} and the fact that $s,v,w$ are bounded at time $t=0$ show that
\begin{equation*}
s,v,w\in C(I(\phi):L^\infty).
\end{equation*}
In addition, $\psi_1,\psi_2\in C(I(\phi):H^{\sigma_0-1})$, and, using the definition \eqref{ny2}, $A_1,A_2\in C(I(\phi):L^4)$ and  $\partial_1^{\sigma_1}\partial_2^{\sigma_2}A_1,\partial_1^{\sigma_1}\partial_2^{\sigma_2}A_2\in C(I(\phi):L^2)$ for any $\sigma_1,\sigma_2\in\mathbb{Z}_+$ with $\sigma_1+\sigma_2\in[1,\sigma_0-1]$. A simple elliptic bootstrap argument, as  in the proof of Proposition \ref{vwonetime}, using the identities \eqref{ny10} for $m=1,2$, shows that 
\begin{equation*}
\partial_1s,\partial_2s\in C(I(\phi):H^{\sigma_0-1}),\partial_1^2v,\partial^2_2v,\partial_1^2w,\partial_2^2w\in C(I(\phi):H^{\sigma_0-2}),
\end{equation*}
as desired.

Finally, using again the identities \eqref{ny10} and \eqref{ny7}, we compute
\begin{equation}\label{ny30}
|Ds|=\big[\sum_{m=1}^2\partial_ms\cdot_\mu\partial_ms\big]^{1/2}=\big[\sum_{m=1}^2|\psi_m|^2\big]^{1/2}=|\psi|.
\end{equation}
It follows from Proposition \ref{regular} (b) that
\begin{equation*}
\begin{split}
&\text{ if }I_+(\phi)\text{ bounded }\text{ then }\| |Ds|\|_{L^4_{x,t}(\R^2\times I_+(\phi))}=\infty,\\
&\text{ if }I_-(\phi)\text{ bounded }\text{ then }\| |Ds|\|_{L^4_{x,t}(\R^2\times I_-(\phi))}=\infty.
\end{split}
\end{equation*}
In particular, the solution $(I(\phi),s)$ is a maximal  solution in the sense of Theorem \ref{thm1}. This completes the existence part of the proof.

\section{Proof of Theorem \ref{thm2}}\label{proofthm2}

Given $f\in \widetilde{H}^{\sigma_0}$ and $Q\in \R^3$ with $Q\cdot_\mu f(0,0)=0$, $Q\cdot_\mu Q=1$, we define the frame $(v,w)$ and the fields $\phi_1,\phi_2\in H^{\sigma_0-1}$ as in Proposition \ref{vwonetime}. We would like to understand first how the functions $\phi_m$ depend  on the choice  of the  point $Q$. Assume $Q'\in\R^3$, $Q'\cdot_\mu f(0,0)=0$, $Q'\cdot_\mu Q'=1$ is  another point and construct the corresponding frame $(v',w')$ and the differentiated fields $\phi_m'$, $m=1,2$. Then
\begin{equation*}
v'=v\cos\chi+w\sin\chi,\qquad w'=-v\sin\chi+w\cos\chi,
\end{equation*} 
for some  $\chi\in C^1(\R^2:\R)$. A simple computation shows that $A'_m=w'\cdot_\mu\partial_mv'=A_m+\partial_m\chi$. However $A_m=A'_m$, since the connection coefficients $A_m$ are defined canonically in \eqref{nj14}, thus
\begin{equation*}
\chi=\mathrm{constant} \quad\text{ on }\R^2. 
\end{equation*}
It follows from the definition \eqref{lo1} that
\begin{equation}\label{lo20}
\phi'_m=z\phi_m,\quad m=1,2,\quad \text{ for some constant }z\in\mathbb{C}\text{ with }|z|=1.
\end{equation}
To summarize, at the level of the fields $\phi_m$, the change of the base point $Q$ leads  to the simple transformation law \eqref{lo20}. Similarly, if $g\in C^3(I:\widetilde{H}^{\sigma_0})$ for some open interval $I\subseteq \R$ and the fields $\psi_m$, $m=0,1,2$, are defined as in Proposition \ref{vwoni} using a  global Coulomb gauge, then the change of the base point $Q$ leads to the transformation
\begin{equation}\label{lo20.5}
\psi'_m=z\psi_m,\quad m=0,1,2,\quad \text{ for some constant }z\in\mathbb{C}\text{ with }|z|=1.
\end{equation}

We can now define the semidistance  $\dot{d}^1$: assume $f,f'\in\HH^{\sigma_0}$, fix $Q,Q'\in\R^3$, $Q\cdot_\mu Q=Q'\cdot_\mu Q'=1$, $Q\cdot_\mu f(0,0)= Q'\cdot_\mu f'(0,0)=0$, and define frames $(v,w)$, $(v',w')$ and differentiated fields $\phi_m,\phi'_m$ as in Proposition \ref{vwonetime}. Then, we define
\begin{equation}\label{di1}
\dot{d}^1(f,f')=\inf_{|z|=1}\big[\|z\phi_1-\phi'_1\|_{L^2}^2+\|z\phi_2-\phi'_2\|_{L^2}^2\big]^{1/2}.
\end{equation} 
Similarly, given an open interval $I$, a point $t_0\in I$, and $g,g'\in C^3(I:\HH^{\sigma_0})$, fix $Q,Q'\in\R^3$, $Q\cdot_\mu Q=Q'\cdot_\mu Q'=1$, $Q\cdot_\mu g(0,0,t_0)= Q'\cdot_\mu g'(0,0,t_0)=0$, and define frames $(v,w)$, $(v',w')$ and differentiated fields $\psi_m,\psi'_m$ as in Proposition \ref{vwoni}. Then, we define
\begin{equation}\label{di1.5}
\begin{split}
\dot{\rho}_I^1(g,g')&=\inf_{|z|=1}\big[\|z\psi_1-\psi'_1\|_{L^\infty_tL^2_x(\R^2\times I)}^2+\|z\psi_2-\psi'_2\|_{L^\infty_tL^2_x(\R^2\times I)}^2\big]^{1/2}\\
&+\inf_{|z|=1}\big[\|z\psi_1-\psi'_1\|_{L^4_{x,t}(\R^2\times I)}^2+\|z\psi_2-\psi'_2\|_{L^4_{x,t}(\R^2\times I)}^2\big]^{1/2}.
\end{split}
\end{equation} 
In view of the discussion above, the definitions \eqref{di1} and  \eqref{di1.5} depend  only on the  functions $f,f'$ and $g,g'$ respectively (in the sense that they do not depend on the choice of the points $Q,Q'$) and clearly define semidistance functions on $\HH^{\sigma_0}$ and $C^3(I:\HH^{\sigma_0})$ respectively.

The identities and the inequality in \eqref{dist1} follow from the definitions and the identities \eqref{nj30} and \eqref{nl18.5} respectively. Theorem \ref{thm2} is also an immediate consequence of Proposition \ref{wellpo}, the construction in subsection \ref{psitos}, and the observation that if $\psi=(\psi_1,\psi_2)$ is a solution of the initial value problem \eqref{tot01}-\eqref{tot1} corresponding to data $\phi=(\phi_1,\phi_2)$, then $z\psi=(z\psi_1,z\psi_2)$ is also a solution corresponding to data $z\phi=(z\phi_1,z\phi_2)$, for any $z\in\mathbb{C}$ with $|z|=1$.

We prove below several additional properties of the semidistance function $\dot{d}^1$. It is not hard to see that the semidistance function $\dot{\rho}^1_I$ also satisfies similar properties.

For $r>0$ and $p\in\mathbb{R}^2$ we define the maps $\delta_r,\tau_p:\widetilde{H}^{\sigma_0}\to\widetilde{H}^{\sigma_0}$,
\begin{equation*}
(\delta_rf)(x)=f(rx),\qquad(\tau_pf)(x)=f(x+p).
\end{equation*}
We define the connected Lie groups $\mathbb{G}_\mu$, $\mu=\pm 1$,
\begin{equation*}
\mathbb{G}_\mu=\{O\in M_3(\mathbb{R}):{}^tO\cdot\eta_\mu\cdot O=\eta_\mu,\,\,\mathrm{det}(O)=1,\,O\cdot{}^t(1,0,0)\in S_\mu\}.
\end{equation*}
Thus $\mathbb{G}_1$ is the orthogonal group $SO(3)$ and $\mathbb{G}_{-1}$ is the Lorentz group $SO(2,1)$. We observe that if $O\in\mathbb{G}_\mu$ and $x,y\in\mathbb{R}^3$ then $Ox\cdot_\mu Oy=x\cdot_\mu y$ and $Ox\times_\mu Oy=O\cdot(x\times_\mu y)$ (this last identity requires $\mathrm{det}(O)=1$). Given $O\in \mathbb{G}_\mu$ we define $R_O:\widetilde{H}^{\sigma_0}\to\widetilde{H}^{\sigma_0}$,
\begin{equation*}
(R_Of)(x)=O\cdot f(x).
\end{equation*}

\begin{proposition}\label{distprop} (a) For any $r\in(0,\infty)$, $p\in\R^2$, $O\in \mathbb{G}_\mu$, and $f,f'\in \widetilde{H}^{\sigma_0}$
\begin{equation}\label{ho4}
\dot{d}^1(\delta_rf,\delta_rf')=\dot{d}^1(f,f'),\quad\dot{d}^1(\tau_pf,\tau_pf')=\dot{d}^1(f,f'),\quad\dot{d}^1(R_Of,R_{O}f')=\dot{d}^1(f,f').
\end{equation}
In addition,
\begin{equation}\label{ho5}
\dot{d}^1(f,f')=0\quad\text{ if and only if }\quad f'=R_Of\quad\text{ for some matrix }O\in\mathbb{G}_\mu.
\end{equation}

(b) The mapping $(f,f')\to\dot{d}^1(f,f')$ is continuous from $(\widetilde{H}^{\sigma_0},d_{\sigma_0})\times(\widetilde{H}^{\sigma_0},d_{\sigma_0})$ to $[0,\infty)$.
\end{proposition}

\begin{proof}[Proof of Proposition \ref{distprop}] The identities \eqref{ho4} are straightforward consequences of the definitions. Also it is easy to check that $\dot{d}^1(f,R_Of)=0$ if $f\in \HH^{\sigma_0}$ and $O\in \mathbb{G}_\mu$. Thus, for \eqref{ho5}, it remains to prove that
\begin{equation}\label{di3}
\text{ if }\quad\dot{d}^1(f,f')=0\quad\text{ then }\quad f'=R_Of\quad\text{ for some }O\in\mathbb{G}_\mu.
\end{equation}
To prove this, we notice that the infimum in \eqref{di1} is attained (since the function in the right-hand side is continuous in $z$). Thus, if $\dot{d}^1(f,f')=0$ then there is $z_0\in\mathbb{C}$ with $|z_0|=1$ such that $\phi'_m=z_0\phi_m$, $m=1,2$. It follows from \eqref{nj17} that $A'_m=A_m$, $m=1,2$. By rotating the frame $(v,w)$ (see the discussion leading to \eqref{lo20}), we may assume that  $z_0=1$. To summarize, we have triples $(f,v,w)$ and $(f',v',w')$ as in Proposition \ref{vwonetime},with the property that the coefficients $\phi_m$ and $A_m$ (see  \eqref{lo1}) agree with the coefficients $\phi'_m$ and $A'_m$ respectively. Then there is a unique matrix $O\in \mathbb{G}_\mu$ such that
\begin{equation*}
R_Of(0,0)=f'(0,0),\quad R_Ov(0,0)=v'(0,0),\quad R_Ow(0,0)=w'(0,0).
\end{equation*}
Let $\delta f=f'-R_Of$, $\delta v=v'-R_Ov$, $\delta w=w'-R_Ow$. Using \eqref{nj30}
\begin{equation*}
\begin{cases}
&\partial_m (\delta f)=(\delta v)\Re(\phi_m)+(\delta w)\Im(\phi_m),\\
&\partial_m (\delta v)=-(\delta f)\mu\Re(\phi_m)+(\delta w)A_m,\\
&\partial_m (\delta w)=-(\delta f)\mu\Im(\phi_m)-(\delta v)A_m
\end{cases}  
\end{equation*}
on $\R^2$, for $m=1,2$. Since $\delta f, \delta v, \delta w$ vanish at $(0,0)$ it follows that $\delta f$ vanishes in $\mathbb{R}^2$, as desired.

We prove now part (b). Since $\dot{d}$ is a semidistance, it suffices to prove that for any $f\in \widetilde{H}^{\sigma_0}$ and $\va>0$ there is $\delta=\delta(f,\va)>0$ such that
\begin{equation}\label{ho20}
\text{ if }f'\in\widetilde{H}^{\sigma_0}\text{ and }d_{\sigma_0}(f,f')\leq\delta\text{ then }\dot{d}^1(f,f')\leq\va.
\end{equation}
Given $f,f'$ as above we fix Coulomb frames $(v,w)$ and $(v',w')$ as  in Proposition \ref{vwonetime}, with $|v(0,0)-v'(0,0)|\lesssim\delta$, and construct the fields $\phi_m$ and $\phi'_m$. Since $f$ is bounded, there is $N=N(f)\geq 1$ such that
\begin{equation}\label{ho21}
\|f\|_{L^\infty}+\|f'\|_{L^\infty}+\|v\|_{L^\infty}+\|v'\|_{L^\infty}+\|w\|_{L^\infty}+\|w'\|_{L^\infty}\leq N.
\end{equation}
Also, since $\phi_m=v\cdot_\mu\partial_mf+iw\cdot_\mu\partial_mf$ and $\phi'_m=v'\cdot_\mu\partial_mf'+iw'\cdot_\mu\partial_mf'$, it follows from \eqref{ho21} and the definition of the distance $d_{\sigma_0}$ (see \eqref{ho3}) that there is $R=R(f,\va)\geq 1$ such that
\begin{equation*}
\sum_{m=1}^2(\|\phi_m\|_{L^2(\{|x|\geq R\})}+\|\phi'_m\|_{L^2(\{|x|\geq R\})})\leq\va/4.
\end{equation*}
Thus, for \eqref{ho20} it suffices to prove that
\begin{equation}\label{ho22}
\sum_{m=1}^2\|\phi'_m-\phi_m\|_{L^\infty(\{|x|\leq R\})}\lesssim\delta,
\end{equation}
where the implicit constant in \eqref{ho22} is allowed to depend on $\va,N,R$.

Recall that the functions $\phi_m\overline{\phi_l}$ and $\phi'_m\overline{\phi'_l}$, $m,l=1,2$, do not depend on the choices of the frames $(v,w)$ and $(v',w')$ respectively (see the discussion at the beginning of section \ref{gauge}). Therefore, by working with local frames $(\widetilde{v},\widetilde{w})$ and $(\widetilde{v'},\widetilde{w'})$ satisfying \eqref{ho21} and  $|\widetilde{v}-\widetilde{v'}|+|\widetilde{w}-\widetilde{w'}|\lesssim\delta$, it follows that
\begin{equation*}
\sum_{m,l=1}^2\|\phi'_m\overline{\phi'_l}-\phi_m\overline{\phi_l}\|_{(L^1\cap L^\infty)(\mathbb{R}^2)}\lesssim \delta.
\end{equation*}
As a consequence,
\begin{equation}\label{ho25}
\sum_{m=1}^2\|A'_m-A_m\|_{L^\infty}\lesssim\delta.
\end{equation}

We prove now the bound \eqref{ho22}. We use the differential equations
\begin{equation*}
\partial_mv=-\mu f\cdot{}^t\partial_mf\cdot\eta_\mu\cdot v+(f\times_\mu v)A_m,\quad m=1,2,
\end{equation*}
see \eqref{nj30}, and the corresponding equations for $v'$. We take the difference of the equations and use \eqref{ho25} and the bounds $|f'-f|\lesssim\delta$, $|\partial_mf'-\partial_mf|\lesssim\delta$ to conclude that
\begin{equation*}
|\partial_m(v'-v)|\lesssim |v'-v|+\delta,\qquad m=1,2.
\end{equation*}
Since $|v'(0,0)-v(0,0)|\lesssim\delta$ it follows that $|v'-v|\lesssim\delta$ in the ball $\{|x|\leq R\}$. The bound \eqref{ho22} follows.
\end{proof}

\section{Reduction to the Davey--Stewartson II equation}\label{link}

It is well known that the Ishimori system, which corresponds to $\ep=-1$, is related to the Davey--Stewartson II equation, at least in the focusing case $\mu=1$. In this section we derive this connection explicitly, starting from our modified spin system, see Proposition \ref{summary}. Recall the formulas
\begin{equation}\label{basn2}
\begin{split}
&q_{12}=\mu\Im(\psi_1\overline{\psi_2}),\qquad A_2=-\nabla^{-1}R_1(q_{12}),\qquad A_1=\nabla^{-1}R_2(q_{12}),\\
&A_0=\mu\sum_{m,l=1}^2 \ep^{m+1}\big[R_lR_m\big(\Re(\overline{\psi_l}\psi_m)\big)+2|\nabla|^{-1} R_l \Im( A_m \psi_m \overline{\psi_l})\big]+\frac{\mu}{2}\sum_{m=1}^2\ep^{m+1}|\psi_m|^2,
\end{split}
\end{equation}
and the equations
\begin{equation} \label{basn3}
\begin{split}
&i \partial_t \psi_m +(\partial_1^2+\ep\partial_2^2) \psi_m = i\mathcal{N}_m,\\
&\mathcal{N}_m =-iA_0\psi_m+\sum_{l=1}^2\ep^{l+1}\big[\psi_l(-q_{ml}+2\partial_mA_l)+\psi_m(-\partial_lA_l+iA_l^2)\big].
\end{split}
\end{equation}

Assume in this section that $\ep=-1$ (the Ishimori system). In this case we expand
\begin{equation*}
\begin{split}
&\mu A_0=\sum_{m,l=1}^2 \ep^{m+1}\big[R_lR_m\big(\Re(\overline{\psi_l}\psi_m)\big)+2|\nabla|^{-1} R_l \Im( A_m \psi_m \overline{\psi_l})\big]+\frac{1}{2}\sum_{m=1}^2\ep^{m+1}|\psi_m|^2\\
&=\sum_{m,l=1}^2 \ep^{m+1}R_lR_m\big(\Re(\overline{\psi_l}\psi_m)\big)+\frac{1}{2}\sum_{m=1}^2\ep^{m+1}|\psi_m|^2+2\sum_{m,l=1}^2 \ep^{m+1}|\nabla|^{-1} R_l [A_m\Im(\psi_m \overline{\psi_l})]\\
&=R_1^2(|\psi_1|^2)-R_2^2(|\psi_2|^2)+\frac{1}{2}|\psi_1|^2-\frac{1}{2}|\psi_2|^2+2\sum_{m,l=1}^2 \ep^{m+1}|\nabla|^{-1} R_l (A_m\cdot\mu q_{ml})\\
&=\frac{1}{2}(R_1^2-R_2^2)(|\psi_1|^2+|\psi_2|^2)+2\mu\nabla^{-1}R_2(A_1q_{12})-2\mu\nabla^{-1}R_1(A_2q_{21})\\
&=\frac{1}{2}(R_1^2-R_2^2)(|\psi_1|^2+|\psi_2|^2)+\mu\nabla^{-2}F
\end{split}
\end{equation*}
where, using the Coulomb condition $\partial_1A_1+\partial_2A_2=0$,
\begin{equation*}
\begin{split}
F&=2\partial_2(A_1q_{12})-2\partial_1(A_2q_{21})=2\partial_2(A_1\partial_1A_2-A_1\partial_2A_1)+2\partial_1(A_2\partial_1A_2-A_2\partial_2A_1)\\
&=2\partial_2[\partial_1(A_1A_2)+A_2\partial_2A_2-A_1\partial_2A_1]+2\partial_1[-\partial_2(A_1A_2)-A_1\partial_1A_1+A_2\partial_1A_2]\\
&=\Delta(A_2^2-A_1^2).
\end{split}
\end{equation*}
Thus
\begin{equation}\label{A0}
A_0=\frac{\mu}{2}(R_1^2-R_2^2)(|\psi_1|^2+|\psi_2|^2)+A_1^2-A_2^2.
\end{equation}

We compute now, using \eqref{basn3}
\begin{equation*}
\begin{split}
i\mathcal{N}_1&=A_0\psi_1+i\psi_1(-\partial_1A_1+iA_1^2+\partial_2A_2-iA_2^2)+i\psi_12\partial_1A_1-i\psi_2(-q_{12}+2\partial_1A_2)\\
&=\psi_1\cdot\frac{\mu}{2}(R_1^2-R_2^2)(|\psi_1|^2+ |\psi_2|^2)-i\psi_2(\partial_2A_1+\partial_1A_2)\\
&=\frac{\mu}{2}\Big[\psi_1\cdot(R_1^2-R_2^2)(|\psi_1|^2+|\psi_2|^2)+i\psi_2\cdot (R_1^2-R_2^2)(2\mu q_{12})\Big]. 
\end{split}
\end{equation*}
Similarly,
\begin{equation*}
\begin{split}
i\mathcal{N}_2&=A_0\psi_2+i\psi_2(-\partial_1A_1+iA_1^2+\partial_2A_2-iA_2^2)+i\psi_1(-q_{21}+2\partial_2A_1)-i\psi_22\partial_2A_2\\
&=\psi_2\cdot\frac{\mu}{2}(R_1^2-R_2^2)(|\psi_1|^2+ |\psi_2|^2)+i\psi_1(\partial_2A_1+\partial_1A_2)\\
&=\frac{\mu}{2}\Big[\psi_2\cdot(R_1^2-R_2^2)(|\psi_1|^2+|\psi_2|^2)-i\psi_1\cdot (R_1^2-R_2^2)(2\mu q_{12})\Big].
\end{split}
\end{equation*}
Thus the system in the first line of \eqref{basn3} becomes
\begin{equation*}
\begin{split}
&(i\partial_0+\partial_1^2-\partial_2^2)\psi_1=f\psi_1+ig\psi_2,\quad (i\partial_0+\partial_1^2-\partial_2^2)\psi_2=f\psi_2-ig\psi_1,\\
&f=(\mu/2)(R_1^2-R_2^2)(|\psi_1|^2+|\psi_2|^2),\quad g=(\mu/2)(R_1^2-R_2^2)(2\mu q_{12}).
\end{split}
\end{equation*}
This system can be decoupled: let $\Phi_{\pm}=\psi_1\pm i\psi_2$. Then
\begin{equation*}
\begin{split}
&(i\partial_0+\partial_1^2-\partial_2^2)\Phi_+=(f+g)\Phi_+,\\
&(i\partial_0+\partial_1^2-\partial_2^2)\Phi_-=(f-g)\Phi_-.
\end{split}
\end{equation*}
Finally, we observe that
\begin{equation*}
\begin{split}
&f+g=(\mu/2)(R_1^2-R_2^2)(|\psi_1|^2+|\psi_2|^2+2\Im(\psi_1\overline{\psi_2}))=(\mu/2)(R_1^2-R_2^2)(\Phi_+\overline{\Phi_+}),\\
&f-g=(\mu/2)(R_1^2-R_2^2)(|\psi_1|^2+|\psi_2|^2-2\Im(\psi_1\overline{\psi_2}))=(\mu/2)(R_1^2-R_2^2)(\Phi_-\overline{\Phi_-}).
\end{split}
\end{equation*}
Therefore we get two decoupled identical equations, for $\Phi=\Phi_{\pm}$
\begin{equation}\label{DS2}
(i\partial_0+\partial_1^2-\partial_2^2)\Phi=(\mu/2)(R_1^2-R_2^2)(|\Phi|^2)\cdot\Phi.
\end{equation}
This is the Davey--Stewartson II equation. 

In other words, in the Ishimori case $\ep=-1$, the modified spin system derived in Proposition \ref{summary} can be simplified algebraically to the Davey--Stewartson II equation \eqref{DS2}, which holds for both functions $\Phi_{\pm}=\psi_1\pm i\psi_2$. This can be used to simplify the analysis of the modified spin system in section \ref{modwell}, in the case $\ep=-1$. Using inverse scattering methods, it is known that the defocusing Davey-Stewartson II equation admits global solutions for $H^{\sigma_0-1}$ data with suitable decay at infinity (see \cite{Su1}--\cite{Su3}). The analysis in subsection \ref{psitos} shows that this leads to global solutions of the original defocusing Ishimori system. More precisely, we have the following large-data global regularity theorem:

\begin{theorem}\label{thmglobal}
Assume $\sigma_0=10$, $\mu=\ep=-1$, and $f\in\widetilde{H}^{\sigma_0}$. Assume, in addition, that $f$ is constant outside a compact set. Then there is a unique global solution $s\in C^3(\mathbb{R}:\widetilde{H}^{\sigma_0})$ of the defocusing Ishimori initial-value problem
\begin{equation*}
\begin{cases}
&\partial_t s= s \times_\mu (s_{11}+\ep s_{22})+s_1\ze_2-\ep s_2\ze_1,\quad\ze_m=-R_m\nabla^{-1}[2 \mu s\cdot_\mu(s_1 \times_\mu s_2)],\\
&s(0)=f.
\end{cases}
\end{equation*}
\end{theorem}

\appendix
\section{Proof of  Proposition \ref{vwoni}}\label{appendix}

In this section we prove Proposition \ref{vwoni}. The idea is to construct the frame $(v,w)$ as the unique solution of the ODE \eqref{nl18.5}. The main step is the following lemma. 

\begin{lemma}\label{CloCoulomb}
Assume $s$, $t_0$, and $Q$ are as in Proposition \ref{vwoni}. Then there is a unique function $v\in C^1(\R^2\times I:\R^3)$ with the properties
\begin{equation}\label{lo6}
\begin{cases}
&\partial_mv=-\mu s\cdot{}^t\partial_ms\cdot\eta_\mu\cdot v+(s\times_\mu v)\widetilde{A}_m,\quad m=0,1,2,\\
&v(0,0,t_0)=Q.
\end{cases}
\end{equation}
In addition $v\cdot_\mu s=0$ and $v\cdot_\mu v=1$ on $\R^2\times I$.
\end{lemma}

\begin{proof}[Proof of Lemma \ref{CloCoulomb}] We may assume $t_0=0$ and observe that the first equation in \eqref{lo6} is consistent with the corresponding equation in \eqref{nl18.5}, since $\Re(\psi_m)=v\cdot_\mu\partial_ms= {}^t\partial_ms\cdot\eta_\mu\cdot v$ and $w=s\times_\mu v$. The equations are of the form $\partial_m v=B_m\cdot v$, for some continuous matrices $B_m$, which gives the uniqueness of $v$. To prove existence, we define first $v(x_1,0,0)$, $x_1\in\R$, by solving the linear homogeneous ODE
\begin{equation}\label{lo7}
\partial_1v=-\mu s\cdot{}^t\partial_1s\cdot\eta_\mu\cdot v+(s\times_\mu v)\widetilde{A}_1,\qquad v(0,0,0)=Q.
\end{equation}
The function $v(.,0,0)$ is a well-defined $C^1$ function on $\R$. Using the equation,
\begin{equation*}
\partial_1(s\cdot_\mu v)=\partial_1({}^ts\cdot\eta_\mu\cdot v)={}^t\partial_1s\cdot\eta_\mu\cdot v+{}^ts\cdot\eta_\mu\cdot[-\mu s\cdot{}^t\partial_1s\cdot\eta_\mu\cdot v+(s\times_\mu v)\widetilde{A}_1]=0,
\end{equation*}
thus $s(x_1,0,0)\cdot_\mu v(x_1,0,0)=0$ for $x_1\in \R$. Using the equation again
\begin{equation*}
\partial_1(v\cdot_\mu v)=\partial_1({}^tv\cdot\eta_\mu\cdot v)=2{}^tv\cdot\eta_\mu\cdot[-\mu s\cdot{}^t\partial_1s\cdot\eta_\mu\cdot v+(s\times_\mu v)\widetilde{A}_1]=0,
\end{equation*}
thus $v(x_1,0,0)\cdot_\mu v(x_1,0,0)=1$ for $x_1\in\R$. 

We extend now $v$ to $\R^2\times\{0\}$ by solving the linear ODE, for every $x_1$ fixed,
\begin{equation}\label{lo8}
\partial_2v=-\mu s\cdot{}^t\partial_2s\cdot\eta_\mu\cdot v+(s\times_\mu v)\widetilde{A}_2,
\end{equation}
with $v(x_1,0,0)$ determined before. The same argument as before shows that 
\begin{equation}\label{nj20}
v\cdot_\mu s=0\quad\text{ and }\quad v\cdot_\mu v=1\quad\text{ on }\R^2\times\{0\}. 
\end{equation}
We prove now the identity 
\begin{equation}\label{lo8.5}
\partial_1v=-\mu s\cdot{}^t\partial_1s\cdot\eta_\mu\cdot v+(s\times_\mu v)\widetilde{A}_1\qquad \text{ on }\R^2\times\{0\}.
\end{equation}
Let $X=\partial_1v+\mu s\cdot{}^t\partial_1s\cdot\eta_\mu\cdot v-(s\times_\mu v)\widetilde{A}_1$. In view of \eqref{lo7}
\begin{equation}\label{lo9}
X(x_1,0,0)=0\quad\text{ for any }x_1\in\R.
\end{equation}
Since $v\cdot_\mu s=0$ and $v\cdot_\mu v=1$ on $\R^2$, it is clear that $v\cdot_\mu X=0$ on $\R^2$. Also
\begin{equation*}
s\cdot_\mu X={}^ts\cdot\eta_\mu\cdot[\partial_1v+\mu s\cdot{}^t\partial_1s\cdot\eta_\mu\cdot v-(s\times_\mu v)\widetilde{A}_1]=\partial_1({}^ts\cdot\eta_\mu\cdot v)=0,
\end{equation*}
on $\R^2\times\{0\}$. Let $w=s\times_\mu v$, and observe that, as a consequence of \eqref{nj20},
\begin{equation}\label{nj21}
w\cdot_\mu v=w\cdot_\mu s=w\cdot_\mu w-1=0,\qquad v\times_\mu w=\mu s,\qquad w\times_\mu s=v.
\end{equation}
For \eqref{lo8.5} it remains to prove that $w\cdot_\mu X=0$, which is equivalent to proving that
\begin{equation}\label{lo11}
{}^tw\cdot\eta_\mu\cdot\partial_1v-\widetilde{A}_1=0\quad\text{ on }\R^2\times\{0\}.
\end{equation}
Using \eqref{lo8}, we have
\begin{equation}\label{lo15}
\partial_2w=\partial_2s\times_\mu v-v\widetilde{A}_2=-\mu s\Im(\phi_2)-v\widetilde{A}_2\quad\text{ on }\R^2\times\{0\},
\end{equation}
where, on $\R^2\times\{0\}$, $\phi_m=v\cdot_\mu\partial_ms+iw\cdot_\mu\partial_ms$, $m=1,2$. Thus, using also \eqref{lo8} and the definition \eqref{nj14},
\begin{equation}\label{blo}
\begin{split}
\partial_2&({}^tw\cdot\eta_\mu\cdot\partial_1v-\widetilde{A}_1)={}^t\partial_2w\cdot\eta_\mu\cdot\partial_1v+{}^tw\cdot\eta_\mu\cdot\partial_1\partial_2v-\partial_2\widetilde{A}_1\\
&=-\mu\Im(\phi_2){}^ts\cdot\eta_\mu\cdot\partial_1v+{}^tw\cdot\eta_\mu\cdot\partial_1(-\mu s\cdot{}^t\partial_2s\cdot\eta_\mu\cdot v+w\widetilde{A}_2)-\partial_2\widetilde{A}_1\\
&=\mu\Im(\phi_2){}^tv\cdot\eta_\mu\cdot\partial_1s-\mu({}^tw\cdot\eta_\mu\cdot\partial_1s)({}^tv\cdot\eta_\mu\cdot \partial_2s)+\partial_1\widetilde{A}_2-\partial_2\widetilde{A}_1\\
&=\mu\Im(\phi_2)\Re(\phi_1)-\mu\Im(\phi_1)\Re(\phi_2)+\partial_1\widetilde{A}_2-\partial_2\widetilde{A}_1\\
&=0.
\end{split}
\end{equation}
The identity \eqref{lo11} follows since $({}^tw\cdot\eta_\mu\cdot\partial_1v-\widetilde{A}_1)(x_1,0,0)=0$, using \eqref{lo9}. This completes the proof of \eqref{lo8.5}.

Finally, we extend the vector $v$ to $\R^2\times I$ by solving the linear ODE, for every $(x_1,x_2)$ fixed,
\begin{equation}\label{lm8}
\partial_0v=-\mu s\cdot{}^t\partial_0s\cdot\eta_\mu\cdot v+(s\times_\mu v)\widetilde{A}_0,
\end{equation}
with $v(x_1,x_2,0)$ defined earlier. As before, it is easy to see that
\begin{equation}\label{lm8.1}
v\cdot_\mu s=0\quad\text{ and }\quad v\cdot_\mu v=1\quad\text{ on }\R^2\times I. 
\end{equation}
It remains to prove the identities
\begin{equation}\label{lm8.5}
\partial_mv=-\mu s\cdot{}^t\partial_ms\cdot\eta_\mu\cdot v+(s\times_\mu v)\widetilde{A}_m\qquad \text{ on }\R^2\times I,
\end{equation}
for $m=1,2$. For this we let $Y_m=\partial_mv+\mu s\cdot{}^t\partial_ms\cdot\eta_\mu\cdot v-(s\times_\mu v)\widetilde{A}_m$, $m=1,2$, and observe that $v\cdot_\mu Y_m=s\cdot_\mu Y_m=0$, as a consequence of \eqref{lm8.1}. A computation similar to \eqref{blo} (using the definition \eqref{nj14} of the coefficients $\widetilde{A}_m$) shows that $w\cdot_\mu Y_m=0$, where $w=s\times_\mu v$. This completes the proof of the lemma.
\end{proof}

We complete now the proof of the proposition. Let $w=s\times_\mu v$, so
\begin{equation}\label{lm9}
w\cdot_\mu v=w\cdot_\mu s=w\cdot_\mu w-1=0,\qquad v\times_\mu w=\mu s,\qquad w\times_\mu s=v,\quad\text{ on }\R^2\times I.
\end{equation}
Let $\psi_m=v\cdot_\mu\partial_ms+iw\cdot_\mu\partial_ms$, $m=0,1,2$. Using \eqref{lm9}, $\partial_ms=v\Re(\psi_m)+w\Im(\psi_m)$, thus using \eqref{lm9} again and Lemma \ref{CloCoulomb}
\begin{equation*}
\partial_m w=\partial_ms\times_\mu v+s\times_\mu\partial_mv=-\mu s\Im(\psi_m)-v\widetilde{A}_m\quad\text{ on }\R^2\times I,
\end{equation*}
for $m=0,1,2$. The identities \eqref{nl18.5} follow. The identities \eqref{nl17} and \eqref{nl18}  follow from \eqref{lm9} and the identities \eqref{nl18.5}.

A simple bootstrap argument using the fact that $s\in C^3(I:\widetilde{H}^{\sigma_0})$, as in the proof  of Proposition \ref{vwonetime}, shows that $v,w\in C^3(\R^2\times I:\R^3)$ and that $\psi_0,\psi_1,\psi_2\in C(I:H^4)$. This completes the proof of the proposition.

\end{document}